\pdfoutput=1
\documentclass[11pt,a4paper]{article}

\usepackage[T1]{fontenc}
\usepackage[utf8]{inputenc}
\usepackage{lmodern}
\usepackage{microtype}
\usepackage{amsmath,amssymb,amsthm}
\usepackage{bm}
\usepackage{booktabs}
\usepackage{graphicx}
\usepackage{tikz}
\usetikzlibrary{arrows.meta,positioning,fit,backgrounds}
\usepackage[round,sort]{natbib}
\usepackage[margin=25mm]{geometry}
\usepackage{url}
\usepackage[hidelinks]{hyperref}

\newcommand{\R}{\mathbb{R}}
\newcommand{\E}{\mathbb{E}}
\newcommand{\Var}{\operatorname{Var}}
\newcommand{\Cov}{\operatorname{Cov}}
\newcommand{\tr}{\operatorname{tr}}
\newcommand{\sgn}{\operatorname{sign}}

\newcommand{\dto}{\xrightarrow{\;d\;}}
\newcommand{\pto}{\xrightarrow{\;p\;}}

\newcommand{\bth}{\bm{\theta}}                 
\newcommand{\bthz}{\bm{\theta}_0}              
\newcommand{\bthtil}{\tilde{\bm{\theta}}}      
\newcommand{\bthor}{\hat{\bm{\theta}}_{\mathrm{or}}}
\newcommand{\bthfe}{\hat{\bm{\theta}}_{\mathrm{fe}}}
\newcommand{\bd}{\bm{d}}                       
\newcommand{\dbar}{\bar{\bm{d}}}               
\newcommand{\Qm}{\bm{Q}}                       
\newcommand{\RD}{R_D}                          

\newcommand{\veps}{\varepsilon}
\newcommand{\bphi}{\bm{\varphi}}               
\newcommand{\bmm}{\bm{m}}                      
\newcommand{\bz}{\bm{z}}                       
\newcommand{\Fm}{\bm{F}}                       
\newcommand{\bb}{\bm{b}}                       
\newcommand{\bh}{\bm{h}}                       
\newcommand{\kap}{\kappa}                      
\newcommand{\gor}{g_{\mathrm{or}}}
\newcommand{\gfe}{g_{\mathrm{fe}}}
\newcommand{\Sor}{\bm{\Sigma}_{\mathrm{or}}}
\newcommand{\Sfe}{\bm{\Sigma}_{\mathrm{fe}}}

\theoremstyle{plain}
\newtheorem{theorem}{Theorem}
\newtheorem{lemma}[theorem]{Lemma}
\newtheorem{corollary}[theorem]{Corollary}
\theoremstyle{definition}
\newtheorem{assumption}{Assumption}
\theoremstyle{remark}
\newtheorem{remark}[theorem]{Remark}

\title{Rank-one variance inflation in polynomial estimation of signal parameters with a
sign-preserving fractional-power basis}
\author{Serhii Zabolotnii\\[4pt]
\small Department of Information, Multimedia Technologies and Design,\\
\small Cherkasy State Business College, Cherkasy 18028, Ukraine\\[2pt]
\small State Scientific Research Institute of Armament and Military Equipment\\
\small Testing and Certification, Cherkasy, Ukraine\\[4pt]
\small \href{mailto:zabolotnii.serhii@csbc.edu.ua}{zabolotnii.serhii@csbc.edu.ua}\quad
ORCID \href{https://orcid.org/0000-0003-0242-2234}{0000-0003-0242-2234}}
\date{}

\begin{document}
\maketitle

\begin{abstract}
Estimating noise moments need not preserve the oracle accuracy of a polynomial estimator.
Earlier plug-in theory establishes oracle accuracy for regression slopes when centered
regressors remove the estimated centering means from their equations. We determine the
covariance cost when that cancellation is absent. For a smooth signal model in a fixed design,
a fixed sign-preserving fractional-power basis and residual centering on the same sample,
the feasible-minus-oracle asymptotic covariance is
$\Sfe-\Sor=c_2(1-\gor)\Qm^{-1}\dbar\dbar^{\!\top}\Qm^{-1}$.
Here $c_2$ is the noise variance, $\gor$ the oracle variance reduction factor, and $\Qm$ and
$\dbar$ the second-moment matrix and mean of the signal gradient. Standard two-step theory
provides the expansion; score identities reduce its correction to a nonnegative scalar times
one outer product. Consistent estimation of the weights has no first-order cost, whereas
estimated centering affects exactly the parameter contrasts with a nonzero projection on
$\Qm^{-1}\dbar$, provided $\gor<1$. An additive intercept absorbs the entire correction;
without it, amplitude, frequency or phase can lose oracle accuracy. The trace loss separates
noise and design through $\gfe=\gor+(1-\gor)\RD$. We derive the limiting distortion of naive
Wald statistics and a corrected covariance estimator. Harmonic-signal simulations support the
directional prediction and document undercoverage from omitting the correction, including under
symmetric noise. Weight-estimation bias for some fractional bases remains substantial at the
sample sizes studied.


\end{abstract}

\section{Introduction}\label{sec:intro}

\subsection{Problem}\label{sec:intro-problem}

Polynomial estimation uses a finite set of noise moments to improve on least squares without
specifying a noise density. In signal models those moments are usually estimated from the same
data used to estimate amplitudes, frequencies or other parameters. The question is which
parameters retain the accuracy available when the moments are known. Consistency of estimated
moments alone does not answer it: their error can still enter the leading distribution of the
parameter estimate.

The closest preceding result already answers this question for a particular slope equation.
\citet{zabolotnii2026sls} identifies optimally weighted second-order least squares with
degree-two polynomial estimation (Theorem~1), and establishes the oracle first-order law of the
feasible slope estimator for a fixed degree $S$ (Proposition~3). The latter result uses a linear
regression with an intercept, errors independent of the regressors, a nonsingular moment matrix
and finite moments through order $2S$. Centering the regressor makes the estimated centering
means cancel exactly. The remaining estimated weights enter an orthogonal block. Both feasible
estimation and the absence of a first-order cost for consistent weight estimates are therefore
established results. Section~\ref{sec:reconcile} compares the statements and assumptions.

Here the target is the covariance cost of residual centering when that cancellation is absent.
For a smooth signal model in a fixed design and a fixed sign-preserving fractional-power basis,
Theorem~\ref{thm:main}, under Assumption~\ref{ass:reg}, gives
\[
\Sfe-\Sor=c_2(1-\gor)\Qm^{-1}\dbar\dbar^{\!\top}\Qm^{-1}.
\]
The noise variance is $c_2$, the oracle variance reduction factor is $\gor$, and $\Qm$ and
$\dbar$ are the limiting second-moment matrix and mean of the signal gradient. The correction
has rank at most one, is positive semidefinite, and lies along $\Qm^{-1}\dbar$. Its scalar
factor is the part of the least-squares variance removed by the oracle score. Increasing the
number of estimated centering moments does not create additional affected directions.

The proof uses the established two-step expansion of \citet{murphy1985estimation} and
\citet{newey1994large}. Neither this expansion nor orthogonality of estimated weights is new.
A low-rank nuisance derivative alone does not imply a positive semidefinite covariance
difference. The model-specific result is the reduction of the full correction, including the
dependence between stages, to the displayed outer product with its explicit coefficient.
It uses an influence difference uncorrelated with the oracle influence
(Remark~\ref{rem:dependence}); different instruments can destroy this property
(Section~\ref{sec:rank-contrast}).

The correction identifies the affected statistical targets. Contrasts orthogonal to
$\Qm^{-1}\dbar$ retain the oracle influence function. An additive intercept absorbs the
entire correction, leaving the other parameters unchanged to first order
(Corollary~\ref{cor:constant}). Without an intercept the loss can fall on the signal parameters
of interest. Its trace is summarized by the design weight $\RD$ in
$\gfe=\gor+(1-\gor)\RD$ (Section~\ref{sec:g}). Corollary~\ref{cor:wald} gives the limiting
law of Wald statistics computed with the oracle covariance and explains why coverage error
persists as the sample size increases whenever the tested contrast is affected.

For the harmonic signal used in the experiments, $\RD=0.2191$; for other signal models on
the same window it ranges from $0.0068$ to $0.2774$ (Table~\ref{tab:classes}). These values
refer to the specified parameter coordinates. Monte Carlo evidence concerns the harmonic
signal; the other models provide deterministic design checks. The experiments assess the
direction of the loss and coverage of corrected regions, while retaining substantial
finite-sample discrepancies for some fractional bases (Section~\ref{sec:s3}).

\begin{remark}[Fixed basis]\label{rem:fixed-basis}
All results are proved for a \emph{fixed} basis parameter; selecting the basis shape from the
data is outside the scope of the paper. Section~\ref{sec:limits} states the conditions under
which first-order asymptotics survive such a selection and explains why the fixed sandwich is
not the covariance of the data-driven estimator.
\end{remark}

\subsection{Related results}\label{sec:related}

The review has five parts. We first consider approaches to signal-parameter estimation beyond the
Gaussian model, then the school's results and the general theory of two-step estimators and
orthogonality on which the proofs rest. The last part concerns the first stage, together with the
main example and the noise distribution used in the experiments.

\subsubsection{Signal-parameter estimation beyond the Gaussian model}\label{sec:related-nongauss}
Beyond the Gaussian model, signal parameters are estimated in three ways, developed most fully
for the harmonic model. Higher-order statistics rely on the fact that the cumulants of Gaussian
noise vanish above the second order. \citet{swami1991cumulant} showed for a sum of harmonics that
one-dimensional slices of the fourth-order cumulant of the noisy signal coincide with the
autocorrelation of the noise-free one, so high-resolution correlation methods can be applied to
cumulants. In this approach non-Gaussianity is a property of the signal and serves to suppress
Gaussian noise; non-Gaussianity of the noise itself is not used to improve accuracy. Fractional
lower-order moments~\citep{shao1993signal,ma1995parameter} are designed for impulsive
$\alpha$-stable noise with infinite variance. Instead of the second moment they use moments of
order $q<\alpha$ and signed powers $\sgn(\veps)|\veps|^{q}$, the same nonlinearity as an element
of basis~\eqref{eq:basis}. Robust statistics builds $M$-estimators with a bounded influence
function~\citep{zoubir2018robust}, and for harmonic regression in heavy-tailed noise a Laplace
periodogram based on least absolute deviations has been proposed~\citep{li2008laplace}. In both
of these approaches the score function is chosen in advance for robustness and is not optimized
for a given noise description. Beyond the harmonic model, \citet{nandi2021estimation} estimated
the parameters of a sum of chirp signals in $\alpha$-stable errors. The polynomial maximization
method lies between these approaches. The noise variance is finite, the noise is described by a
finite set of moments, and the score weights are optimal for that description within the class
of linear combinations of basis elements. Because the weights depend on the noise moments, the
consequences of estimating these moments become a meaningful question.

\subsubsection{Results of the Kunchenko school}\label{sec:related-school}
For signal-parameter estimation the school's results developed in several directions. For a
harmonic signal in noise with skewness and excess kurtosis, oracle variance reduction factors
were obtained for non-identically distributed sample values. In particular, at $S=2$ the joint
estimates of amplitude, frequency and phase share the factor $1-\gamma_3^2/2$ regardless of the
design sums~\citep{palahin2016joint}. Joint estimation of the amplitude together with the noise
variance and skewness was also considered, but in a design with a zero mean signal gradient,
where no accuracy is lost (Section~\ref{sec:reconcile}). The method has been extended to
correlated non-Gaussian processes~\citep{vokorokos2017parameters}, to linear regression with
asymmetric~\citep{zabolotnii2018polynomial} and symmetric~\citep{zabolotnii2021estimating}
errors, where the factor $g_S$ is the same for all regression coefficients, and to nonlinear
regression~\citep{zabolotnii2024application}. A continuous generalization of the basis, the
signed-parity fractional-power polynomial with a shape parameter~$\alpha$
\citep{zabolotnii2026patp}, defines the class of score functions used in this paper.

\subsubsection{Effect of estimated nuisance parameters}\label{sec:related-nuisance}
\citet{randles1982asymptotic} gave a general answer to the question of how replacing an unknown
parameter $\lambda$ by its estimate $\hat\lambda$ changes the limit law of a statistic. There
$\hat\lambda$ is replaced by a variable $\gamma$, and the \emph{limiting mean} of the statistic
is differentiated. The substitution leaves the limit law unchanged if and only if the derivative
of the limiting mean is zero; otherwise an additional term appears in the limit law.
\citet{pierce1982asymptotic} considered the case in which the limiting mean does not depend
on~$\lambda$ and $\hat\lambda$ is asymptotically efficient. The corresponding limiting
quantities are then uncorrelated, and the variance \emph{decreases}. For estimating equations in
which the nuisance parameter is found from (partial) score equations, \citet{lok2024how}
systematized this effect together with a consistent sandwich estimator. In our problem the
conditions of that result fail: the centering means are estimated by sample means rather than by
score equations, and the increment is nonnegative.

The econometric literature describes the same correction from a computational point of view. It
gives an additive correction to the second-stage covariance through the cross derivatives of
the two stages' scores~\citep{murphy1985estimation}, the condition under which standard errors
with generated regressors remain valid~\citep{pagan1984econometric}, and the sandwich covariance
of a two-step extremum estimator~\citep[Sect.~6]{newey1994large}. Stacking the moment conditions
of both stages into one system lets any sequential estimator be treated as a method-of-moments
estimator~\citep{newey1984method} in the sense of \citet{hansen1982large}. In that
representation, replacing weight matrices by consistent estimates does not change the asymptotic
covariance, whereas plugging in an estimated first-stage \emph{parameter} adds a separate term
to it. In the statistical literature the same device is known as the calculus of $M$-estimators
with stacked equations~\citep{stefanski2002calculus}; \citet{boe2024practical} apply it to
two-stage regression as a cheaper alternative to the bootstrap. Its semiparametric
generalization is the asymptotic variance formula based on a pathwise
derivative~\citep{newey1994asymptotic}, extended to generated regressors by
\citet{hahn2013asymptotic}. If the dimension of the first stage is large relative to~$\sqrt n$,
a first-order bias appears~\citep{cattaneo2019two}. Here the first stage is finite-dimensional,
and there is no such bias.

These results provide the tools used here. We do not claim the general sandwich correction,
orthogonality of consistently estimated weights, or a rank bound for a linear influence map.
The additional calculation identifies the correction for residual centering in the present
model and evaluates its sign, direction and coefficient. The feasible-slope result of
\citet{zabolotnii2026sls} already covers a setting where the centering contribution cancels;
Section~\ref{sec:reconcile} separates that result from Theorem~\ref{thm:main}.

\subsubsection{Neyman orthogonality and adaptive estimation}
In the modern literature, the condition under which a nuisance parameter does not affect the
limit law is called Neyman orthogonality: the Gateaux derivative of the expected score in the
direction of the nuisance is zero~\citep{chernozhukov2018double}.
\citet{chernozhukov2022locally} give a general construction of locally robust moment functions,
whose mean is locally insensitive to the first stage. In the finite-dimensional case this
condition coincides with a zero block of cross derivatives in sequential method-of-moments
estimation~\citep{hansen1982large} and with the criterion of \citet{randles1982asymptotic}. For
estimating equations, a class that includes the polynomial maximization method, parameter
orthogonality has been extended to insensitivity to the nuisance parameter, and under this
condition estimating the nuisance costs no efficiency~\citep{jorgensen2004parameter}. Higher
moments of non-Gaussian residuals appear in this theory too: in partially linear regression,
second-order orthogonal moments exist if and only if the regressor residual is
non-Gaussian~\citep{mackey2018orthogonal}. Semiparametric theory has a stronger result.
\citet{stone1975adaptive} estimated the score of an unknown symmetric density from the same
sample and attained the parametric bound, and \citet{bickel1982adaptive} stated a general
criterion of adaptivity: the efficient score for the parameter of interest must be orthogonal to
the nuisance tangent space. The monograph of \citet{bickel1993efficient} gives a complete account
of this theory. One may ask why estimating a few moments affects the asymptotic covariance when,
under adaptivity, estimating an entire score function does not. Section~\ref{sec:split} gives the
answer. Adaptivity concerns the \emph{shape} of the score under symmetry, whereas here the shape
is fixed, and a finite-dimensional centering vector is estimated at rate $\sqrt{n}$ in a model
\emph{without an intercept}.

\subsubsection{Model, baseline fit and noise distribution}
The first stage is a nonlinear least-squares (NLS) fit of the model; the classical results cited
above give the conditions for its $\sqrt n$-consistency. For the harmonic
signal~\eqref{eq:harmonic} used in the simulations, this is a sine fit. In our notation it has
\emph{three} parameters, because the model has no intercept; the four-parameter form of the
IEEE~1057 standard also includes a constant term. Such fits are practically efficient for
Gaussian and quantization noise~\citep{handel2000properties}. Outside that class their
efficiency is not guaranteed, and this is where a polynomial estimator can improve accuracy.
\citet{kollar2005improved} proposed a refined definition of the best-fitting sine for testing
analog-to-digital converters, and \citet{carbone2014rigorous} gave a rigorous analysis of the fit
on quantized data. The metrological result closest to our first stage is the estimation of noise
power by three- and four-parameter sine-fitting algorithms~\citep{belega2012noise}. There the
noise variance is estimated from the fit; here the signed centering means of fractional powers of
the residuals are estimated as well. In nonlinear regression theory, a neighboring line of work is
$L_p$-estimation~\citep{ivanov2008asymptotic}: its score $p|\veps|^{p-1}\sgn\veps$ has the same
functional form as a single element of the sign-preserving basis. More generally,
\citet{ivanov2007consistency} proved the consistency of $M$-estimators for an arbitrary even
Lipschitz loss function. The settings of these works differ from ours. They consider Gaussian
dependent noise, continuous time, a single minimized functional and a score without centering by
estimated moments. Moreover, the condition of evenness and Lipschitz continuity does not cover the
odd fractional-power score, and the condition of symmetry and unimodality does not cover
asymmetric errors. In the experiments the noise follows the asymmetric exponential power
distribution~\citep{zhu2009properties}, with different left and right tail exponents; its
skewness and kurtosis have closed forms. \citet{matsui2016fractional} give ways to compute the
fractional absolute moments $\E|X-\mu|^{\gamma}$, $1<\gamma<2$, of heavy-tailed laws. The vector
$\bb$ also needs moments of negative orders, which these formulas do not cover
(Section~\ref{sec:neg-orders}). The degenerate mirror case of the basis is the
expectile~\citep{newey1987asymmetric}; a comparison with expectiles is beyond the scope of this
paper.

\section{Model, noise, basis and notation}\label{sec:model}

We observe
\begin{equation}\label{eq:model}
y_i \;=\; r_i(\bthz) + \veps_i,
\qquad
r_i(\bth) \;=\; r(t_i;\bth),
\qquad i=1,\dots,n,
\end{equation}
where $\bth\in\R^p$ is the vector of signal parameters and $r(t;\bth)$ is a known function, twice
continuously differentiable in $\bth$. The nodes $t_1,\dots,t_n$ form a fixed design on a finite
window, and $\veps_i$ are errors with zero mean and finite variance $c_2=\E\veps^2$ (the precise
requirements are in Assumption~\ref{ass:reg}). The noise description is incomplete: no density is
specified, and only a finite set of moments is used. The main example is the harmonic signal with
an envelope
\begin{equation}\label{eq:harmonic}
r(t;\bth) \;=\; A\,e(t)\cos(\Omega t + \Phi),
\qquad
\bth=(A,\Omega,\Phi)^{\!\top}\in\R^3,
\end{equation}
where $e(\cdot)$ is a known envelope. The same class includes, among others, a sum of several
harmonics, a damped sinusoid, a chirp signal and a smooth pulse with unknown amplitude, position
and duration; Section~\ref{sec:g} gives the design characteristics of these models.

We denote the signal gradient and its sample moments by
\begin{equation}\label{eq:design}
\bd_i=\frac{\partial r_i(\bthz)}{\partial\bth},
\qquad
\Qm_n=\frac1n\sum_{i=1}^n \bd_i\bd_i^{\!\top},
\qquad
\dbar_n=\frac1n\sum_{i=1}^n \bd_i,
\end{equation}
and assume $\Qm_n\to\Qm>0$, $\dbar_n\to\dbar$. The window mean of the gradient, $\dbar$, plays a
key role in the paper, because the increment is concentrated along the direction $\Qm^{-1}\dbar$.

We use the sign-preserving fractional-power basis (form~B of \citealp{zabolotnii2026patp}) with a
\emph{fixed} shape parameter~$\rho$:
\begin{equation}\label{eq:basis}
q_1=1,\quad q_j=p_j(\rho),\ j=2,\dots,S,
\qquad
\varphi_j(\veps)=\sgn(\veps)\,|\veps|^{q_j}.
\end{equation}
The exponents follow the quadratic map of form~B,
\[
p_j(\rho)=\frac1j+\Bigl(4-j-\frac3j\Bigr)\rho+\Bigl(2j-4+\frac2j\Bigr)\rho^{2},
\qquad 0\le\rho\le1,
\]
so that $p_j(0)=1/j$ (a fractional basis), $p_j(1)=j$ (signed integer powers), and $p_j(1/2)=1$
for every $j$, where all basis columns coincide. The signed centering means, the centered basis, its correlation matrix, the vector of derivative
moments, the optimal weights, the sensitivity and the centered score are defined as
\begin{equation}\label{eq:score}
\begin{aligned}
\bmm&=\E[\bphi(\veps)], &\quad \bz(\veps)&=\bphi(\veps)-\bmm, &\quad \Fm&=\E[\bz\bz^{\!\top}],\\
\bb&=\bigl(q_j\,\E|\veps|^{q_j-1}\bigr)_{j=1}^{S}, &\quad \bh&=\Fm^{-1}\bb, &\quad
\kap&=\bh^{\!\top}\bb=\bb^{\!\top}\Fm^{-1}\bb,\\
\psi(\veps)&=\bh^{\!\top}\bz(\veps).
\end{aligned}
\end{equation}
Optimality of the weights gives two identities that we use below:
\begin{equation}\label{eq:identities}
\Var[\psi(\veps)]=\bh^{\!\top}\Fm\bh=\bh^{\!\top}\bb=\kap,
\qquad
\Cov[\psi(\veps),\veps]=\bh^{\!\top}\Fm\bm{e}_1=b_1=1,
\end{equation}
where the second equality uses $q_1=1$, $\varphi_1(\veps)=\veps$ and $b_1=1$.

All matrices $\bm{\Sigma}$ below are asymptotic covariances of $\sqrt{n}(\hat\bth-\bthz)$; the
covariance of the estimate itself is $\bm{\Sigma}/n$.

\subsection{Regularity conditions}\label{sec:assumptions}

All statements of the paper are proved under one set of conditions. We state it here because
Lemma~\ref{lem:mm} already needs it. Items~1--3 concern the model and the design, items~4--6 the
basis and the noise moments, and items~7--8 the two stages of the estimator.

\begin{assumption}\label{ass:reg}
For a fixed $\rho$, the following conditions suffice.
\begin{enumerate}
\item $\bthz$ is an interior point, $r_i(\bth)$ is twice continuously differentiable in a
      neighborhood $\mathcal N$ of $\bthz$, and the first and second design derivatives are
      uniformly bounded:
      $\sup_i\lVert\bd_i\rVert<\infty$ and
      $\sup_i\sup_{\bth\in\mathcal N}\bigl\lVert\partial^2 r_i(\bth)/
      \partial\bth\,\partial\bth^{\!\top}\bigr\rVert<\infty$. For the harmonic
      signal~\eqref{eq:harmonic} on a finite window with a bounded envelope this holds
      automatically; in particular, it implies $\max_i\lVert\bd_i\rVert/\sqrt n\to0$.
\item $\Qm_n\to\Qm>0$, $\dbar_n\to\dbar$.
\item The errors are independent and identically distributed, independent of the fixed design,
      $\E\veps=0$, $0<c_2=\E\veps^2<\infty$.
\item $\rho$ lies outside the band where the exponents coalesce, the exponents are distinct,
      $\Fm>0$, and the condition number of the population $\Fm$ is finite.
\item The moments required for $\Fm$ and $\bmm$ exist; in particular,
      $\E|\veps|^{q_j+q_k}<\infty$ for all pairs of basis functions.
\item For the expansion of residual centering, $\E|\veps|^{q_j-1}<\infty$ and
      $\E|\veps|^{2q_j}<\infty$; there is no atom at zero, and when $q_j<1$ the noise has a
      density that is bounded in a neighborhood of zero.
\item The preliminary NLS estimate is $\sqrt n$-consistent; the constraints and safeguards are
      inactive with probability tending to one; the refinement step is accepted in full.
\item The estimates $\hat\Fm,\hat\bb,\hat\bh,\hat\kap$ are consistent, and the stochastic
      expansion is uniform in a neighborhood of $\bthz$.
\end{enumerate}
\end{assumption}

Section~\ref{sec:limits} discusses what item~6 implies for the admissible exponents $q_j$, in
particular that the bound $q_j>1/2$ is neither necessary nor sufficient. Items~7 and~8 are
high-level conditions in the sense of \citet[Sect.~6]{newey1994large}: they concern the estimator
as implemented, including its safeguards, and the paper does not derive them from primitive
conditions on the noise law.

\section{The oracle estimator and its covariance}\label{sec:oracle}

The oracle estimator uses the true $\bmm$, $\Fm$, $\bb$, $\bh$ and $\kap$. Starting from a
preliminary estimate $\bthtil$ obtained by constrained nonlinear least squares, it performs one
refinement step:
\begin{equation}\label{eq:oracle-step}
\bthor=\bthtil+\tilde\Qm_n^{-1}\kap^{-1}\,\frac1n\sum_{i=1}^n \tilde\bd_i\,
\psi\bigl(y_i-r_i(\bthtil)\bigr),
\qquad
\tilde\bd_i=\frac{\partial r_i(\bthtil)}{\partial\bth},\quad
\tilde\Qm_n=\frac1n\sum_{i=1}^n\tilde\bd_i\tilde\bd_i^{\!\top}.
\end{equation}
The gradient is evaluated at the preliminary estimate, so the step is implementable; by item~1
of Assumption~\ref{ass:reg}, replacing $\tilde\bd_i$ by $\bd_i$ changes it only by
$o_p(n^{-1/2})$. Since the preliminary estimate is $\sqrt{n}$-consistent and
$\E[\psi'(\veps)]=\bh^{\!\top}\bb=\kap$, the refinement step cancels the influence of the first
stage, so that
\begin{equation}\label{eq:oracle-influence}
\sqrt{n}\,(\bthor-\bthz)=\Qm^{-1}\kap^{-1}\,\frac1{\sqrt n}\sum_{i=1}^n \bd_i\,\psi(\veps_i)+o_p(1),
\qquad
\Sor=\Qm^{-1}/\kap .
\end{equation}
The covariance of the NLS estimate itself is $c_2\Qm^{-1}$, so the oracle variance reduction
factor is
\begin{equation}\label{eq:gor}
\gor=\frac{1}{c_2\,\kap}.
\end{equation}

Quantity~\eqref{eq:gor} is computed from the true moment characteristics; it is the same oracle
level as in \citet{palahin2016joint}. The increment derived below is measured relative to it.

\section{The feasible scheme as a two-step estimator}\label{sec:feasible}

The feasible scheme differs from~\eqref{eq:oracle-step} only in where the moment characteristics
come from (Fig.~\ref{fig:pipeline}):
\begin{enumerate}
\item constrained NLS gives $\bthtil$;
\item the residuals $\hat\veps_i=y_i-r_i(\bthtil)$ are centered by their sample mean,
      $\hat\veps_i^{c}=\hat\veps_i-n^{-1}\sum_k\hat\veps_k$;
\item the centered residuals give the sample counterparts of~\eqref{eq:score},
      \[
      \hat\bmm=\frac1n\sum_i\bphi(\hat\veps_i^{c}),\qquad
      \hat\Fm=\frac1n\sum_i\hat\bz_i\hat\bz_i^{\!\top},\qquad
      \hat b_j=\frac{q_j}{n}\sum_i|\hat\veps_i^{c}|^{q_j-1},
      \]
      where $\hat\bz_i=\bphi(\hat\veps_i^{c})-\hat\bmm$, and $|\hat\veps_i^{c}|$ in $\hat b_j$ is
      replaced by $\max(|\hat\veps_i^{c}|,\epsilon_n)$ when $q_j<1$
      (Section~\ref{sec:neg-orders}); then $\hat\bh=\hat\Fm^{-1}\hat\bb$,
      $\hat\kap=\hat\bh^{\!\top}\hat\bb$ and $\hat c_2=n^{-1}\sum_i(\hat\veps_i^{c})^2$;
\item one refinement step is performed with the feasible score
      $\hat\psi(\veps)=\hat\bh^{\!\top}[\bphi(\veps)-\hat\bmm]$.
\end{enumerate}

\begin{figure}[t]
\centering
\begin{tikzpicture}[
  node distance=6mm and 7mm,
  box/.style={draw,rounded corners=2pt,align=center,inner sep=4pt,minimum height=9mm,font=\small\sffamily},
  free/.style={box,fill=black!4},
  paid/.style={box,fill=black!18,very thick},
]
\node[box] (nls) {constrained NLS\\$\bthtil$};
\node[box,right=of nls] (res) {residuals\\$\hat\veps_i^{c}$};
\node[free,above right=3mm and 9mm of res] (fb) {$\hat\Fm,\ \hat\bb,\ \hat\bh,\ \hat\kap$\\\scriptsize orthogonal block};
\node[paid,below right=3mm and 9mm of res] (mm) {$\hat\bmm$\\\scriptsize centering means};
\coordinate (mid) at (mm.east |- res.east);
\node[box,right=7mm of mid] (step) {one refinement\\step};
\node[box,right=of step] (out) {$\bthfe$};
\draw[->] (nls) -- (res);
\draw[->] (res) -- (fb);
\draw[->] (res) -- (mm);
\draw[->] (fb) -- (step);
\draw[->] (mm) -- (step);
\draw[->] (step) -- (out);
\draw[->,rounded corners=6pt] (nls.south) |- ([yshift=-5mm]mm.south) -| (step.south);
\end{tikzpicture}
\caption{Feasible estimation scheme. The light block holds the moment characteristics whose
estimation is orthogonal and costs no first-order accuracy; the dark block holds the signed
centering means, whose estimation produces increment~\eqref{eq:main}}
\label{fig:pipeline}
\end{figure}
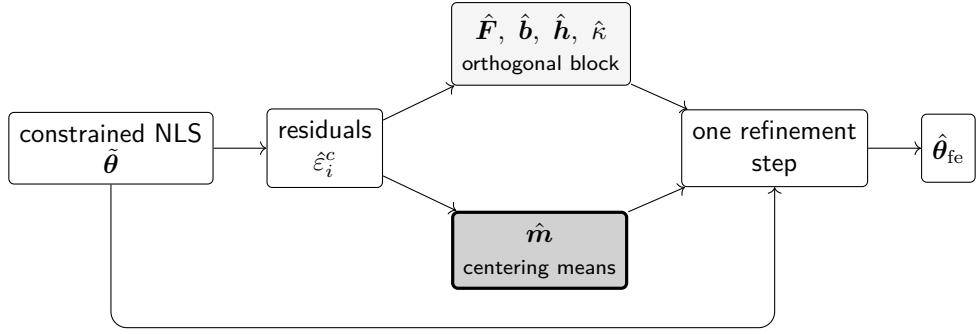

The covariance of such an estimator follows from the general two-step sandwich formula. Let the
main estimating equation and the equation for the nuisance parameter be
$U_n(\bth,\bm\eta)=0$ and $V_n(\bthtil,\bm\eta)=0$, and let
\[
\bm{A}=\E\Bigl[\frac{\partial U_i}{\partial\bth^{\!\top}}\Bigr],\quad
\bm{C}=\E\Bigl[\frac{\partial U_i}{\partial\bm\eta^{\!\top}}\Bigr],\quad
\bm{H}=\E\Bigl[\frac{\partial V_i}{\partial\bm\eta^{\!\top}}\Bigr],\quad
\bm{D}=\E\Bigl[\frac{\partial V_i}{\partial\bthtil^{\!\top}}\Bigr].
\]
If $\sqrt n(\bthtil-\bthz)=n^{-1/2}\sum_i\bm{\ell}_i+o_p(1)$, the efficient score and the
sandwich are
\begin{equation}\label{eq:twostep}
\bm{s}_i^{\mathrm{eff}}=\bm{u}_i-\bm{C}\bm{H}^{-1}(\bm{v}_i+\bm{D}\bm{\ell}_i),
\qquad
\bm{\Sigma}_{\text{2-step}}=\bm{A}^{-1}\Var(\bm{s}_i^{\mathrm{eff}})\bm{A}^{-\top}.
\end{equation}
This is the form of \citet[Sect.~6]{newey1994large} (Section~\ref{sec:related-nuisance}).

In the feasible scheme the nuisance parameter that matters to first order is $\bm\eta=\bmm$; the
block $(\Fm,\bb,\bh,\kap)$ is left out because it is orthogonal (Section~\ref{sec:split}). The
two equations are
\[
U_i(\bth,\bmm)=\bd_i\,\bh^{\!\top}\bigl[\bphi\bigl(y_i-r_i(\bth)\bigr)-\bmm\bigr],
\qquad
V_i(\bthtil,\bmm)=\bphi(\hat\veps_i^{c})-\bmm,
\]
where $\hat\veps_i^{c}$ depends on $\bthtil$, and the expectations in $\bm A,\bm C,\bm H,\bm D$
are averaged over the design. Then $\bm A=-\kap\Qm$, $\bm C=-\dbar\,\bh^{\!\top}$ and
$\bm H=-\bm I$. The averaged expectation of $\partial V_i/\partial\bthtil^{\!\top}$ is
$-\bb\,n^{-1}\sum_i(\bd_i-\dbar_n)^{\!\top}=\bm 0$, so $\bm D=\bm 0$, and Lemma~\ref{lem:mm}
below shows that the empirical counterpart of this term vanishes as well. The same lemma gives
$\bm v_i=\bz(\veps_i)-\bb\,\veps_i$, where the term $-\bb\,\veps_i$ comes from centering by the
sample mean. Substituting into~\eqref{eq:twostep} and using $\bh^{\!\top}\bb=\kap$,
\[
\bm s_i^{\mathrm{eff}}=\bd_i\,\psi(\veps_i)-\dbar\bigl[\psi(\veps_i)-\kap\,\veps_i\bigr],
\qquad
\bm\Sigma_{\text{2-step}}=\Qm^{-1}\Bigl[\lim_{n\to\infty}\frac1n\sum_i\Var(\bm s_i^{\mathrm{eff}})\Bigr]\Qm^{-1}\big/\kap^{2},
\]
where $\bm s_i^{\mathrm{eff}}$ is the influence term $\bm\chi_i$ of Theorem~\ref{thm:main} below.

\begin{remark}[The sign of the increment is not fixed in advance]\label{rem:pierce}
The condition of the result of \citet{pierce1982asymptotic} (Section~\ref{sec:related-nuisance}),
asymptotic efficiency of the first stage, fails here: constrained NLS is not efficient in
non-Gaussian noise. The cross term in~\eqref{eq:twostep} therefore does not vanish, and only the
general statement of \citet{randles1982asymptotic} applies. In our problem, positive
semidefiniteness of the difference follows from Theorem~\ref{thm:main}, not from the general
theory.
\end{remark}

\section{Main result}\label{sec:main}

\subsection{Influence of the estimated centering moments}\label{sec:influence}

\begin{lemma}\label{lem:mm}
Under Assumption~\ref{ass:reg}, for every exponent $q_j$,
\begin{equation}\label{eq:mm-influence}
\sqrt{n}\,(\hat\bmm-\bmm)=\frac1{\sqrt n}\sum_{i=1}^n\bigl[\bz(\veps_i)-\bb\,\veps_i\bigr]+o_p(1),
\end{equation}
where the term generated by the preliminary NLS estimate is absent and the contribution of the
first basis column ($q_1=1$) is exactly zero.
\end{lemma}

The proof is given in Appendix~\ref{app:lem-mm}; its key step is as follows. In the expansion of
$\hat\bmm$ about $\bmm$, the coefficient of $\sqrt n(\bthtil-\bthz)$ is the empirical matrix
$\bm{A}_n=n^{-1}\sum_i\bphi'(\veps_i)(\bd_i-\dbar_n)^{\!\top}$. The identity
$\sum_i(\bd_i-\dbar_n)=\bm 0$ annihilates only its population part, and the remainder tends to
zero by the law of large numbers. The contribution of the preliminary NLS estimate therefore
vanishes: centering the residuals on the same sample asymptotically orthogonalizes the
signed-moment estimator with respect to the unknown location parameter.

\subsection{Separating the sources of the increment}\label{sec:split}

We write the feasible score as $\hat\psi(\veps)=\hat\bh^{\!\top}[\bphi(\veps)-\hat\bmm]$ and
consider the derivatives of its expectation with respect to each block of nuisance parameters.

For the block $(\Fm,\bb,\bh,\kap)$ the derivative is zero, because the score is centered:
\begin{equation}\label{eq:orth}
\E[\bz(\veps)]=\bm 0
\quad\Longrightarrow\quad
\frac{\partial}{\partial \bh^{\!\top}}\,\E\bigl[\bh^{\!\top}\bz(\veps)\bigr]=\E[\bz(\veps)]^{\!\top}=\bm 0^{\!\top}.
\end{equation}
Hence the product of $(\hat\bh-\bh)$ and the centered empirical score is $o_p(n^{-1/2})$;
likewise, the error of $\hat\kap$ multiplies a score of order $O_p(n^{-1/2})$. In the language of
\citet{chernozhukov2018double} this block is \emph{Neyman orthogonal}; in the terms of
\citet{hansen1982large} it has $\bm{C}=\bm 0$, and consistency of the estimates suffices.

For the block of centering means the derivative is not zero:
\begin{equation}\label{eq:nonorth}
\frac{\partial\psi}{\partial\bmm^{\!\top}}=-\bh^{\!\top},
\qquad
\frac{\partial}{\partial\bmm^{\!\top}}\,\E\bigl[\bd\,\psi\bigr]=-\,\dbar\,\bh^{\!\top}\ne\bm 0
\ \text{ if } \dbar\ne\bm 0 .
\end{equation}
The first-order loss of accuracy thus comes precisely from estimating the signed centering
moments. The feasible scheme estimates two different blocks of nuisance parameters, and only one
of them affects the asymptotic covariance.

\subsection{The rank-one theorem}\label{sec:theorem}

\begin{theorem}\label{thm:main}
Under Assumption~\ref{ass:reg}, for a fixed basis parameter and centering on the same sample,
\begin{equation}\label{eq:feas-influence}
\sqrt{n}\,(\bthfe-\bthz)=\Qm^{-1}\kap^{-1}\frac1{\sqrt n}\sum_{i=1}^n\bm\chi_i+o_p(1),
\qquad
\bm\chi_i=(\bd_i-\dbar)\,\psi(\veps_i)+\dbar\,\kap\,\veps_i,
\end{equation}
and the asymptotic covariance of the feasible estimator is
\begin{equation}\label{eq:main}
\Sfe=\Sor+\Bigl(c_2-\frac1\kap\Bigr)\Qm^{-1}\dbar\dbar^{\!\top}\Qm^{-1}
=\Sor+c_2\,(1-\gor)\,\Qm^{-1}\dbar\dbar^{\!\top}\Qm^{-1}.
\end{equation}
The difference $\Sfe-\Sor$ is positive semidefinite and has rank at most one.
\end{theorem}

\begin{proof}
Substitute Lemma~\ref{lem:mm} into the expansion of the one-step refinement. The feasible and
oracle scores differ by
\[
\hat\psi(\veps)-\psi(\veps)=(\hat\bh-\bh)^{\!\top}\bz(\veps)-\hat\bh^{\!\top}(\hat\bmm-\bmm).
\]
For a single observation the first term need not be $o_p(n^{-1/2})$; no root-$n$ rate for
$\hat\bh-\bh$ is assumed. It is negligible in the scaled estimating equation, where it enters as
$\bm Z_n(\hat\bh-\bh)$ with $\bm Z_n=n^{-1/2}\sum_i\bd_i\bz(\veps_i)^{\!\top}$. This matrix has
zero mean by~\eqref{eq:orth} and bounded second moments by items~1, 2 and~5 of
Assumption~\ref{ass:reg}, so $\bm Z_n=O_p(1)$, while $\hat\bh-\bh=o_p(1)$ by item~8. In the second
term $\hat\bh$ may be replaced by $\bh$, because $\sqrt n(\hat\bmm-\bmm)=O_p(1)$ by
Lemma~\ref{lem:mm}. Hence
\[
\frac1{\sqrt n}\sum_{i=1}^n\bd_i\bigl[\hat\psi(\veps_i)-\psi(\veps_i)\bigr]
=-\,\dbar_n\,\bh^{\!\top}\sqrt n\,(\hat\bmm-\bmm)+o_p(1).
\]
To first order, replacing $\bmm$ by $\hat\bmm$ thus shifts the centered score by a \emph{scalar}
that is the same for all $i$, and the estimating equation is shifted by this scalar times the
\emph{fixed} vector $\dbar$. Hence the rank of the increment is at most one, not~$S$: a vector of $S$
centering means is estimated, but it enters the equation only through the single scalar linear
combination $\bh^{\!\top}\bmm$. Together with~\eqref{eq:mm-influence}, this
gives~\eqref{eq:feas-influence}.

The summands $\bm\chi_i$ are independent but not identically distributed, because $\bd_i$
varies with~$i$, so the ``meat'' matrix is the design average of their covariances. By
identities~\eqref{eq:identities}, $\Var(\bm\chi_i)=\kap(\bd_i-\dbar)(\bd_i-\dbar)^{\!\top}
+\kap\bigl[(\bd_i-\dbar)\dbar^{\!\top}+\dbar(\bd_i-\dbar)^{\!\top}\bigr]
+\kap^2c_2\,\dbar\dbar^{\!\top}$. The cross term averages to zero because $\dbar_n\to\dbar$,
and the first term averages to $\kap(\Qm-\dbar\dbar^{\!\top})$ in the limit, so
\[
\bm{B}_{\mathrm{fe}}=\lim_{n\to\infty}\frac1n\sum_{i=1}^n\Var(\bm\chi_i)
=\kap\,(\Qm-\dbar\dbar^{\!\top})+\kap^2 c_2\,\dbar\dbar^{\!\top},
\]
whence $\Sfe=\Qm^{-1}\bm{B}_{\mathrm{fe}}\Qm^{-1}/\kap^{2}$ equals~\eqref{eq:main}. The second
equality in~\eqref{eq:main} substitutes $\gor=1/(c_2\kap)$ from~\eqref{eq:gor}. The matrix
$\Qm^{-1}\dbar\dbar^{\!\top}\Qm^{-1}=(\Qm^{-1}\dbar)(\Qm^{-1}\dbar)^{\!\top}$ is an outer product,
hence positive semidefinite with rank at most one. The factor $c_2(1-\gor)$ is nonnegative
because $\gor\le1$: by the Cauchy--Schwarz inequality and identities~\eqref{eq:identities},
$1=\Cov[\psi(\veps),\veps]^2\le\Var[\psi(\veps)]\,\Var\veps=\kap\,c_2$.
\end{proof}

\begin{remark}[One term accounts for the dependence between the stages]\label{rem:dependence}
The difference between the feasible and oracle estimators has the influence representation
\begin{equation}\label{eq:diff-influence}
\sqrt{n}\,(\bthfe-\bthor)=\Qm^{-1}\dbar\,\frac1{\sqrt n}\sum_{i=1}^n
\Bigl[\veps_i-\frac{\psi(\veps_i)}{\kap}\Bigr]+o_p(1),
\end{equation}
which is asymptotically uncorrelated with the oracle influence, and its covariance equals the term
in~\eqref{eq:main} exactly (Lemma~\ref{lem:diff} in Appendix~\ref{app:diff}). The dependence
between the first stage and the moment stage, caused by the use of one sample, is therefore fully
accounted for by this term, and no separate correction for it is needed.
\end{remark}

\begin{remark}[The rank is that of the channel, not the number of estimated moments]\label{rem:rank}
Theorem~\ref{thm:main} can be read in the general form~\eqref{eq:twostep}. The difference between
the efficient and oracle scores equals $-\bm{C}\bm{H}^{-1}(\bm{v}_i+\bm{D}\bm{\ell}_i)$ and lies
in the range of the cross Jacobian $\bm{C}$. If it is uncorrelated with the oracle score (as it
is here, by Remark~\ref{rem:dependence}), then
\[
\Sfe-\Sor=\bm{A}^{-1}\bm{C}\bm{H}^{-1}\,\Var(\bm{v}_i+\bm{D}\bm{\ell}_i)\,\bm{H}^{-\top}\bm{C}^{\!\top}\bm{A}^{-\top},
\qquad
\operatorname{rank}(\Sfe-\Sor)\le\operatorname{rank}\bm{C}.
\]
In our scheme, by~\eqref{eq:nonorth}, the cross Jacobian with respect to the block of centering
means equals $-\dbar\,\bh^{\!\top}$. It is an outer product, with rank one for every $S$ when $\dbar\ne\bm0$ and rank zero otherwise,
although all $S$ sign-preserving centering means are estimated. The one-dimensional influence
channel is a property of how the estimator is built: the weight vector $\bh=\Fm^{-1}\bb$
aggregates the centering means into one scalar $\bh^{\!\top}\bmm$ before they enter the
estimating equation. Entering each centering mean separately does not by itself raise the rank.
If the moment conditions have the form $\bd_i\otimes\bz(\veps_i)$ (instruments shared by all
basis components) and the weight is the efficient matrix, that matrix reproduces the same
aggregation through $\bh$, and the increment again has rank at most one. A rank above one arises only
when different basis components have different instruments. In that case, however, the
uncorrelatedness condition above also fails, and positive semidefiniteness of $\Sfe-\Sor$ is no
longer guaranteed (Section~\ref{sec:rank-contrast} gives a numerical example).
\end{remark}

\subsection{Consequences}\label{sec:consequences}

\begin{corollary}[Direction of the increment]\label{cor:direction}
For any $\bm{a}\in\R^p$,
\[
\bm{a}^{\!\top}(\Sfe-\Sor)\bm{a}=c_2(1-\gor)\,\bigl(\bm{a}^{\!\top}\Qm^{-1}\dbar\bigr)^2 ,
\]
and this quantity is zero for all $\bm{a}\perp\Qm^{-1}\dbar$: linear combinations of the
parameters orthogonal to $\Qm^{-1}\dbar$ are estimated with oracle accuracy even in the feasible
regime. If, in addition, $\dbar\ne\bm 0$ and $\gor<1$, the increment is nonzero and concentrated
in the \emph{single} direction $\bm{u}=\Qm^{-1}\dbar/\lVert\Qm^{-1}\dbar\rVert$. If $\dbar=\bm 0$
or $\gor=1$, the increment matrix is zero, there is no increment in any direction, and no
distinguished direction exists.
\end{corollary}

The identity follows from~\eqref{eq:main} and the symmetry of $\Qm^{-1}$. For
$\bm{a}\perp\Qm^{-1}\dbar$ a stronger statement holds: by~\eqref{eq:diff-influence},
$\sqrt n\,\bm{a}^{\!\top}(\bthfe-\bthor)=o_p(1)$. Such a linear combination of the feasible and
oracle estimators thus coincides to first order, rather than merely having the same asymptotic
variance. Appendix~\ref{app:direction} gives the full proof.

\begin{corollary}[Oracle covariance in feasible Wald statistics]\label{cor:wald}
Under the conditions of Theorem~\ref{thm:main}, let $\bm a\ne\bm0$ be a fixed contrast and
let $\hat\Sor\pto\Sor=\Qm^{-1}/\kap$. A confidence interval for
$\bm a^{\!\top}\bthz$ using the oracle standard error has limiting coverage
\begin{equation}\label{eq:wald-contrast}
2\Phi\!\left(\frac{z_{1-\alpha/2}}{\sqrt{1+\delta_{\bm a}}}\right)-1,
\qquad
\delta_{\bm a}=\frac{c_2(1-\gor)(\bm a^{\!\top}\Qm^{-1}\dbar)^2}
{\bm a^{\!\top}\Sor\bm a},
\end{equation}
where $0<\alpha<1$, $\Phi$ is the standard normal distribution function and $z_{1-\alpha/2}$
its quantile. This coverage is strictly below $1-\alpha$ exactly when $\delta_{\bm a}>0$.
For the full parameter vector, the naive Wald statistic satisfies
\begin{equation}\label{eq:wald-joint}
n(\bthfe-\bthz)^{\!\top}\hat\Sor^{-1}(\bthfe-\bthz)
\dto\sum_{j=1}^{p-1}Z_j^2+(1+\tau)Z_p^2,
\qquad \tau=\frac{1-\gor}{\gor}\,\pi,
\quad \pi=\dbar^{\!\top}\Qm^{-1}\dbar,
\end{equation}
with independent standard normal $Z_j$ and an empty sum for $p=1$. Thus the usual
chi-squared region undercovers when $\tau>0$. Replacing $\hat\Sor$ by a consistent estimate
of $\Sfe$ restores the standard normal and chi-squared limits.
\end{corollary}

\begin{proof}
We first establish the normal limit. The summands $\bm\chi_i$ in~\eqref{eq:feas-influence} are
independent, and by items~1, 3 and~5 of Assumption~\ref{ass:reg} the vectors $\bd_i$ are
uniformly bounded, the errors are identically distributed, and $\psi(\veps)$ and $\veps$ have
finite second moments. Hence the Lindeberg condition holds for every fixed linear combination of the
$\bm\chi_i$, and the limiting covariance is nonsingular, because
$\bm{B}_{\mathrm{fe}}-\kap\Qm=\kap(\kap c_2-1)\dbar\dbar^{\!\top}$ is positive semidefinite. By
the Lindeberg--Feller theorem and the Cram\'er--Wold device, $\sqrt n\,(\bthfe-\bthz)$ converges
in distribution to a centered normal vector with covariance $\Sfe$.
Corollary~\ref{cor:direction} gives the variance of the contrast. Slutsky's theorem then yields
a normal limit with variance $1+\delta_{\bm a}$ after division by the oracle standard error,
which proves~\eqref{eq:wald-contrast}. For the joint statistic, whiten the limit by
$\Sor^{-1/2}$. Its covariance is $\bm I+\lambda\bm w\bm w^{\!\top}$, with
$\lambda=c_2(1-\gor)$ and $\bm w=\Sor^{-1/2}\Qm^{-1}\dbar$.
All eigenvalues are unity except one, equal to $1+\lambda\lVert\bm w\rVert^2=1+\tau$.
Orthogonal diagonalization proves~\eqref{eq:wald-joint}. For $\tau>0$ the limiting quadratic
form strictly stochastically exceeds a chi-squared variable with $p$ degrees of freedom.
The final assertion follows by whitening with the feasible covariance instead.
\end{proof}

This is a consequence of the covariance identity, not a separate general theory of Wald
inference. For a coordinate parameter, set $\bm a=\bm e_j$: it loses oracle precision if and
only if $\gor<1$ and $(\Qm^{-1}\dbar)_j\ne0$. These limits do not guarantee accurate
finite-sample coverage when the remainders in Assumption~\ref{ass:reg} remain large.

\begin{corollary}[Four sufficient conditions for oracle equivalence]\label{cor:equiv}
First-order oracle equivalence, $\sqrt n(\bthfe-\bthor)\pto\bm 0$, holds in each of the
following cases:
\begin{enumerate}
\item \textbf{known centering means}: $\bmm$ is supplied externally, and only
      $\Fm,\bb,\bh,\kap$ are estimated; for a symmetric law this corresponds to setting
      $\bmm=\bm0$ explicitly, but such a procedure uses an additional assumption of symmetry;
\item \textbf{centered design gradient}: $\dbar=\bm 0$;
\item \textbf{neutral score}: $c_2=1/\kap$, that is, $\gor=1$; then by~\eqref{eq:identities}
      $\veps=\psi/\kap$ in $L_2$, and the right-hand side of~\eqref{eq:diff-influence} is zero;
      this case covers Gaussian neutrality;
\item \textbf{general orthogonality}: $\bm{C}=\bm 0$ in~\eqref{eq:twostep}, or the exact
      cancellation $\bm{C}\bm{H}^{-1}(\bm{v}_i+\bm{D}\bm{\ell}_i)=\bm 0$.
\end{enumerate}
If the centering means are estimated, $\dbar\ne\bm 0$ and $\gor<1$, growth of $n$ alone does not
remove the relative gap: both covariances are divided by $n$, and the factor in~\eqref{eq:main}
remains.
\end{corollary}

The proof is given in Appendix~\ref{app:equiv}. In cases~2--4 the leading term
in~\eqref{eq:diff-influence} vanishes; in case~1 the procedure itself changes, and
Lemma~\ref{lem:mm} is not needed.

Conditions~1--4 eliminate the nuisance contribution to the leading influence, but not always
by making its derivative zero. Orthogonality of the derivative is the familiar sufficient
mechanism in \citet{randles1982asymptotic}. In the neutral-score case,
$\veps-\psi/\kap=0$ in $L_2$ even if $\bm C=-\dbar\bh^{\!\top}\ne\bm0$; known centering means
instead remove the estimation error itself. These mechanisms should be distinguished.

\begin{remark}[Noise symmetry is not a fifth condition]\label{rem:symmetry}
For a symmetric law all sign-preserving centering means are zero, so one might expect that they
need not be estimated. However, $\bmm=\bm 0$ does not make the error $\hat\bmm-\bmm$ zero, and
this error enters the equation through the mean design gradient $\dbar$, which does not depend on
the parity of the noise. In the trace metric this can be seen from~\eqref{eq:gfe}. The increment
equals $(1-\gor)\RD$, the product of a factor that depends only on the noise and a factor that
depends only on the design, and the parity of the noise does not enter it. The increment
vanishes only when $\gor=1$ (condition~3) or $\dbar=\bm 0$ (condition~2). Symmetry can be
exploited only through the procedure of condition~1, which sets $\bmm=\bm 0$ explicitly and
thereby introduces symmetry as an assumption. On the same design ($\RD=0.2191$) the largest
increment among all scenarios of Section~\ref{sec:numeric} belongs to the symmetric law S075
($0.169$ at $S=3$), whereas for the asymmetric A1L it is only $0.0013$ at $S=2$; the polynomial
degree affects both values only through $\gor$. The zero increment in scenarios A0 and S200 is
due to Gaussian neutrality (condition~3), not to symmetry. Table~\ref{tab:coverage-s3} confirms
this on the behavior of the estimator.

In practice, whether condition~2 holds is decided by the geometry of the window, not by the noise
distribution. With a constant envelope and a record covering a whole number of periods, the node
averages of $\cos(\Omega t+\Phi)$ and $\sin(\Omega t+\Phi)$ are zero, and so are two of the three
coordinates of $\dbar$; the third coordinate, which contains the factor $t$, is relatively small.
On the benchmark window $t\in[0.01,\,2]$ the period is incomplete, and $\RD$ varies almost
thirtyfold across designs (Table~\ref{tab:rd}).
\end{remark}

\begin{remark}[The increment direction as a projection of the constant]\label{rem:projection}
Since $\dbar_n=n^{-1}\sum_i\bd_i\cdot1$ and $\Qm_n=n^{-1}\sum_i\bd_i\bd_i^{\!\top}$, the vector
$\Qm_n^{-1}\dbar_n$ solves the least-squares problem
\begin{equation}\label{eq:projection}
\Qm_n^{-1}\dbar_n=\arg\min_{\bm a\in\R^p}\frac1n\sum_{i=1}^n\bigl(1-\bd_i^{\!\top}\bm a\bigr)^2,
\qquad
\pi_n=\dbar_n^{\!\top}\Qm_n^{-1}\dbar_n
=1-\min_{\bm a}\frac1n\sum_{i=1}^n\bigl(1-\bd_i^{\!\top}\bm a\bigr)^2\in[0,1].
\end{equation}
It is therefore the vector of coefficients of the projection of the constant function onto the
linear span of the signal gradient components, and $\pi_n$ is the squared norm of this projection
in the design metric; $\pi$ denotes the limit of $\pi_n$. The increment is thus determined by how
well small changes of the signal parameters reproduce a shift of all observations by a constant.
Condition~2 of Corollary~\ref{cor:equiv} means that the constant is orthogonal to the gradient
components ($\pi=0$). The opposite extreme, $\pi=1$, is treated in Corollary~\ref{cor:constant}.
From~\eqref{eq:main}, $\tr\bigl[\Qm(\Sfe-\Sor)\bigr]=c_2(1-\gor)\,\pi$, so in the metric
weighted by~$\Qm$ the variance reduction factor equals $\gor+(1-\gor)\,\pi/p$. Unlike the weight
$\RD$ in~\eqref{eq:gfe}, the quantity $\pi$ is invariant under smooth one-to-one
reparametrization, in particular under a change of measurement units.
\end{remark}

\begin{corollary}[Model with a constant term]\label{cor:constant}
Suppose there is a vector $\bm a_0\in\R^p$ with $\bd_i^{\!\top}\bm a_0=1$ for all $i$. Then
$\Qm^{-1}\dbar=\bm a_0$ and $\pi=1$. In particular, if $r_i(\bth)=s_i(\bm\vartheta)+C$ and
$\bth=(\bm\vartheta^{\!\top},C)^{\!\top}$, then $\Qm^{-1}\dbar=\bm e_C$ and
\[
\Sfe-\Sor=c_2(1-\gor)\,\bm e_C\bm e_C^{\!\top}:
\]
the asymptotic variance of the estimate of the constant term grows by $c_2(1-\gor)$, and the
estimates of $\bm\vartheta$ coincide with the oracle ones to first order.
\end{corollary}

\begin{proof}
By assumption, $\dbar_n=n^{-1}\sum_i\bd_i\,(\bd_i^{\!\top}\bm a_0)=\Qm_n\bm a_0$, whence
$\Qm_n^{-1}\dbar_n=\bm a_0$ and $\pi_n=\dbar_n^{\!\top}\bm a_0=n^{-1}\sum_i\bd_i^{\!\top}\bm a_0=1$;
passing to the limit gives the same for $\Qm$ and $\dbar$. For an additive constant term
$\partial r_i/\partial C=1$, so $\bm a_0=\bm e_C$, and the formula for the increment follows
from~\eqref{eq:main}. For the last statement, take any $\bm a$ whose coordinate at~$C$ is zero.
Then $\bm a^{\!\top}\Qm^{-1}\dbar=\bm a^{\!\top}\bm e_C=0$, and
representation~\eqref{eq:diff-influence} gives
$\sqrt n\,\bm a^{\!\top}(\bthfe-\bthor)=o_p(1)$, which is the stronger statement following
Corollary~\ref{cor:direction}; equality of asymptotic variances alone would not suffice.
\end{proof}

An example is the four-parameter sine fit of the IEEE~1057 standard, which contains a constant
term in addition to $(A,\Omega,\Phi)$. In the feasible regime the amplitude, frequency and phase
are estimated with first-order oracle accuracy, and the increment falls on the estimate of the
constant term. In the limiting case where the signal is the constant itself ($p=1$, $\Qm=1$),
$\RD=1$ and $\gfe=1$. If all centering means are estimated by sample means from the same
residuals, the values of the centered score sum to zero identically. The refinement step is then
zero, and the feasible estimate coincides with the sample mean for every $S$. The gain survives
if the centering means are given a priori (condition~1 of Corollary~\ref{cor:equiv}), and
whether this is possible depends on the noise distribution. For symmetric noise the centering
means of odd powers are zero, and the classical polynomial maximization method at $S=3$ uses
exactly this assumption. For asymmetric noise at $S=2$ the centering mean of the square is the
noise variance. It is unknown a priori, and estimating it from the same residuals brings back the
sample mean. A simulation check confirms this ($20\,000$ replications, integer-power basis,
weights estimated from the residuals in both variants). With estimated centering the refinement
step is exactly zero in both cases. With centering given a priori, at $n=5000$ the empirical
factor is $0.3008$ for the uniform law at $S=3$ (population value $0.3000$) and $0.5036$ for the
centered exponential law at $S=2$ (population value $0.5000$); the latter uses the true variance,
that is, oracle information.

\begin{remark}[Sample splitting does not help]\label{rem:splitting}
Estimating $\bmm$ independently on another subsample does not give oracle equivalence. It removes
the useful covariance between the score and the moment estimator on the same sample and can
increase the variance. Cross-fitting is needed to control overfitting with complex nuisance
models~\citep{chernozhukov2018double}, but it does not replace the construction of an orthogonal
score.
\end{remark}

\subsection{Relation to the plug-in equivalence result}\label{sec:reconcile}

Theorem~1 and Proposition~3 of \citet{zabolotnii2026sls} answer different questions.
Theorem~1 identifies the optimally weighted SLS slope equation with degree-two PMM, through a
change of coordinates of the moment vector. Proposition~3 establishes consistency and the
oracle limiting variance of the feasible slope for any fixed degree $S$. Its assumptions
A1--A4 require independent identically distributed errors independent of the random design,
finite error and design moments through order $2S$, a nonsingular polynomial moment matrix,
and an intercept. The proof in Supplement~S3 needs consistency of the estimated weights, not
their root-$n$ asymptotic normality. The degree-three efficiency reserve in that paper concerns
enlarging the moment span; it is a separate result from the cost of estimating its moments.

The decisive identity in its feasible-slope proof is
\[
\sum_i(x_i-\bar x)\,\hat{\bm\alpha}_{1:S}=\bm0.
\]
Thus the estimated centering means disappear from the slope equation exactly. The remaining
weight error multiplies a mean-zero sample moment and is negligible to first order. That
argument already establishes the orthogonal-weight mechanism used in~\eqref{eq:orth}.
It does not evaluate a nonzero feasible-minus-oracle covariance correction for the full
parameter vector, nor does its displayed slope limit assert oracle accuracy of the intercept.

The present theorem evaluates that correction when the centering contribution survives.
Its assumptions are different rather than uniformly weaker: the design is fixed, the signal
can be nonlinear and need not contain an intercept, and the basis is~\eqref{eq:basis} rather
than ordinary powers $\veps^j$. In particular, an even signed power is not an ordinary even
power. Assumption~\ref{ass:reg} requires the fractional derivative moments, local conditions
near zero, a root-$n$ preliminary estimate and a uniform stochastic expansion. The latter
conditions are explicitly assumed, not derived here from the moment conditions of the earlier
paper. The refinement is~\eqref{eq:oracle-step}; the earlier construction uses a Newton step
or a locally convergent solution of its slope equation. Consequently, neither theorem is
claimed to contain the other as a formal special case.

Their statistical mechanisms agree. In the present fixed-design setting, an additive intercept
gives $\Qm^{-1}\dbar=\bm e_C$ (Corollary~\ref{cor:constant}), so the correction vanishes on
every slope contrast but may remain on the intercept. Without an intercept, a contrast loses
oracle accuracy precisely when $\gor<1$ and its projection on $\Qm^{-1}\dbar$ is nonzero
(Corollary~\ref{cor:direction}). The absence of an intercept alone is not sufficient:
$\dbar=\bm0$ or a neutral score can still remove the correction. This distinction is relevant
to the harmonic model~\eqref{eq:harmonic}, whose amplitude, frequency and phase need not lie
in the unaffected subspace.

Corollary~\ref{cor:equiv} also explains why, in the school's early work, estimating the amplitude
of a harmonic signal jointly with the noise variance and skewness did not reduce accuracy
(Section~\ref{sec:related-school}). The variance of the amplitude estimate obtained there is
reproduced as an identity in $\gamma_3$ only when the node averages of $\cos(\Omega t_i+\Phi)$ and
$\cos^2(\Omega t_i+\Phi)$ equal zero and $1/2$, respectively, that is, under condition~2 of
Corollary~\ref{cor:equiv}. The absence of accuracy loss in that work is therefore due to the
design, not to the way the noise parameters were estimated.

\section{Consequence for the variance reduction factor}\label{sec:g}

The monograph of \citet{kunchenko2002polynomial} compares estimators by the trace metric relative
to least squares; for nonlinear regression the reference is NLS~\citep{zabolotnii2024application}.
Since the asymptotic covariance of NLS is $c_2\Qm^{-1}$ (Section~\ref{sec:oracle}), the variance
reduction factor of an estimator with asymptotic covariance $\bm\Sigma$ is
$g=\tr\bm\Sigma/(c_2\tr\Qm^{-1})$. For $\bm\Sigma=\Sor$ this gives $\gor$ of~\eqref{eq:gor}, and
$\gfe$ denotes the same ratio for $\bm\Sigma=\Sfe$. We define the design weight
\begin{equation}\label{eq:RD}
\RD=\frac{\lVert\Qm^{-1}\dbar\rVert^{2}}{\tr\Qm^{-1}} .
\end{equation}

\begin{corollary}\label{cor:g}
In the trace metric, the feasible variance reduction factor
$\gfe=\tr\Sfe/(c_2\tr\Qm^{-1})$ equals
\begin{equation}\label{eq:gfe}
\gfe=\gor+(1-\gor)\,\RD .
\end{equation}
\end{corollary}

\begin{proof}
Symmetry of $\Qm^{-1}$ gives $\tr\bigl(\Qm^{-1}\dbar\dbar^{\!\top}\Qm^{-1}\bigr)
=\lVert\Qm^{-1}\dbar\rVert^{2}$, and~\eqref{eq:oracle-influence} and~\eqref{eq:gor} give
$\tr\Sor=\gor\,c_2\tr\Qm^{-1}$. It remains to take the trace of both sides of~\eqref{eq:main} and
divide by $c_2\tr\Qm^{-1}>0$ (items~2--3 of Assumption~\ref{ass:reg}): the left-hand side
becomes $\gfe$ by the definition of $g$ above, and the rank-one term becomes $(1-\gor)\RD$
by~\eqref{eq:RD}.
\end{proof}

Formula~\eqref{eq:gfe} has a simple interpretation. The feasible factor interpolates between the
oracle value and one, and \emph{the interpolation weight does not depend on the noise}. The noise
distribution affects the result only through $\gor$, and the experimental conditions only through
$\RD$. The share of the oracle gain that is lost is therefore set by the design, and it can be
reduced by changing the observation window without changing the estimator or the noise.
Table~\ref{tab:rd} gives the values of $\RD$ for seven designs.

\begin{table}[t]
\caption{Design weight $\RD$ for seven designs. For each row the value was computed under $15$
noise distributions and agrees to machine precision: $\RD$ does not depend on the noise}
\label{tab:rd}
\centering
\begin{tabular}{lccc}
\toprule
Design (window, envelope) & $n=200$ & $n=1000$ & $n=5000$ \\
\midrule
$[0.01,\,2]$, envelope $0.3$ (main)       & $0.2243$ & $0.2200$ & $0.2191$ \\
$[0.01,\,2]$, no envelope                 & $0.2374$ & $0.2344$ & $0.2338$ \\
$[0.01,\,2]$, envelope $0.6$              & $0.2208$ & $0.2152$ & $0.2141$ \\
$[0.01,\,1]$, envelope $0.3$              & $0.0119$ & $0.0118$ & $0.0117$ \\
$[0.01,\,4]$, envelope $0.3$              & $0.3386$ & $0.3432$ & $0.3441$ \\
$[0.01,\,8]$, envelope $0.3$              & $0.0739$ & $0.0738$ & $0.0738$ \\
$[2,\,4]$, envelope $0.3$                 & $0.0990$ & $0.0964$ & $0.0959$ \\
\bottomrule
\end{tabular}
\end{table}

For the main design, $\RD=0.2191$: centering on the same sample keeps about $78.1\%$ of the
oracle gain, and about $21.9\%$ is lost. Across designs $\RD$ ranges from $0.0117$ to $0.3441$,
almost thirtyfold.

\subsection{Design characteristics of other signal classes}\label{sec:classes}

The weight $\RD$ and the quantity $\pi$ of Remark~\ref{rem:projection} depend on the signal model
only through its gradient, so they can be computed for any model of class~\eqref{eq:model}
without simulation. Table~\ref{tab:classes} gives the results for eight models on the window
$t\in[0.01,\,2]$ at $n=5000$.

\begin{table}[t]
\caption{Design characteristics of eight signal models, window $[0.01,\,2]$, $n=5000$. $\RD$ is
the weight in the trace metric~\eqref{eq:RD}, $\pi$ is the squared norm of the projection of the
constant~\eqref{eq:projection}, and $\pi/p$ is the share of the oracle gain lost in the metric
weighted by $\Qm$. The computation is deterministic}
\label{tab:classes}
\centering
\begin{tabular}{lcccc}
\toprule
Signal model & $p$ & $\RD$ & $\pi$ & $\pi/p$ \\
\midrule
harmonic with envelope~\eqref{eq:harmonic}              & $3$ & $0.2191$ & $0.9786$ & $0.3262$ \\
harmonic with envelope and constant term                & $4$ & $0.0068$ & $1.0000$ & $0.2500$ \\
sum of two harmonics                                    & $6$ & $0.0096$ & $1.0000$ & $0.1667$ \\
damped sinusoid                                         & $4$ & $0.0092$ & $0.9995$ & $0.2499$ \\
linear frequency-modulated (chirp) signal               & $4$ & $0.2331$ & $0.5151$ & $0.1288$ \\
Gaussian pulse (amplitude, position, width)             & $3$ & $0.0946$ & $0.7888$ & $0.2629$ \\
linear model without intercept ($t$, $t^2$)             & $2$ & $0.2774$ & $0.8932$ & $0.4466$ \\
constant                                                & $1$ & $1.0000$ & $1.0000$ & $1.0000$ \\
\bottomrule
\end{tabular}
\end{table}

The table supports three conclusions. First, for the model with a constant term $\pi=1$, and the
direction $\Qm^{-1}\dbar$ coincides with the unit vector of that term to within
$1.4\cdot10^{-14}$, as Corollary~\ref{cor:constant} predicts. Second, on a window shorter than
the period, the gradient of a smooth signal reproduces the constant almost exactly. For the
harmonic signal $\pi=0.9786$, and for the sum of two harmonics and the damped sinusoid $\pi$
differs from one by less than $0.001$. In the metric weighted by $\Qm$ the lost share of the
oracle gain is therefore close to $1/p$; for the harmonic signal it is $0.326$, against $0.219$
in the trace metric. Third, the values of $\RD$ for different models differ by more than an order
of magnitude at similar values of $\pi$. The reason is that the trace metric adds up the
variances of parameters measured in different units and depends on their scale, whereas $\pi$
does not depend on the parametrization. Values of $\RD$ for different models should be compared
with this dependence in mind.

For each of the eight models we checked identity~\eqref{eq:gfe} under $15$ noise distributions
and three sample sizes. In the $312$ of $360$ cells where the ratio is defined, the maximum
deviation of $\bigl(\gfe-\gor\bigr)/(1-\gor)$ from $\RD$ is $1.0\cdot10^{-14}$. This verifies the
algebra of Theorem~\ref{thm:main} for gradients of different structure. The finite-sample
behavior of the estimator for these models has not been studied: the simulations of
Section~\ref{sec:numeric} use only the harmonic signal~\eqref{eq:harmonic}.

\section{Scope and limits}\label{sec:limits}

Assumption~\ref{ass:reg} collects the hypotheses under which all statements are proved. Here we
consider three questions: which signal models it covers, which exponents $q_j$ it actually
admits, and which of the formulas survive if the basis parameter is selected from the data.

\subsection{The class of signal models}\label{sec:signal-class}
Items~1 and~2 of Assumption~\ref{ass:reg} hold for any model that is twice continuously
differentiable in the parameters, with derivatives bounded on the window and a nonsingular limit
matrix $\Qm$. Three cases fall outside this class. The first is a signal that is not smooth in a
parameter, such as a rectangular pulse with an unknown onset time: the gradient does not exist,
and the expansion of the refinement step loses its meaning. The second is an observation window
that grows without bound. For a harmonic signal the frequency estimate then converges faster
than the amplitude and phase estimates~\citep{walker1971estimation}, the matrix $\Qm_n$ with the
single normalization $\sqrt n$ has no nonsingular limit, and item~2 fails. The statements of
this paper concern a fixed window with increasingly dense nodes; the behavior of the increment
under componentwise normalization is not studied here. The third is an ill-conditioned matrix
$\Qm$. The assumption then holds formally, but first-order asymptotics set in only at larger $n$.
Among the models of Table~\ref{tab:classes}, the condition number of $\Qm$ is $18$ for the
harmonic signal and $2.4\cdot10^{3}$ for the sum of two harmonics.

\subsection{Negative orders of derivative moments}\label{sec:neg-orders}
The order of the moment in $b_j=q_j\,\E|\veps|^{q_j-1}$ can be negative. If the density is
positive and bounded near zero, the moment exists only for $q_j>0$, and the unsmoothed sample
moment has finite variance at rate $\sqrt n$ only for $q_j>1/2$. At the same time the estimate of
$\bb$ enters the first-order covariance orthogonally (see~\eqref{eq:orth}), so consistency of
$\hat\bh$ and $\hat\kap$ suffices, and rate $\sqrt n$ is not needed. The current implementation,
with smoothing-truncation at the level $\epsilon_n=\hat\sigma\,n^{-3/4}$, is consistent for
$q_j-1>-1$. If the smoothing bias is also required to be negligible at rate $\sqrt n$, one needs
in addition $\sqrt n\,\epsilon_n^{q_j}\to0$, that is, $q_j>2/3$ for the exponent $3/4$.

The rule $q_j>1/2$ is therefore correct for finite variance of the unsmoothed estimator of a
derivative moment. It is, however, \emph{neither necessary} for consistency of the weights
\emph{nor sufficient} for the smoothing bias of the current implementation to be negligible at
rate $\sqrt n$. All statements of the paper take this distinction into account. At $\rho=1/2$
all exponents equal one, the basis columns coincide, and the basis degenerates; at $\rho<1/2$ at
least one exponent is below one, and the vector $\bb$ requires absolute moments of negative order.
For $S=3$ at $\rho=0.3$, the value used in scenarios S075--S200 of Section~\ref{sec:numeric},
$q_3\approx0.573$: the condition $q_j>1/2$ holds, but $q_j>2/3$ does not.
Section~\ref{sec:s3} discusses the consequences.

\subsection{Selecting the basis parameter from the data}\label{sec:rho-select}
Formulas~\eqref{eq:main}--\eqref{eq:gfe} are proved for a \emph{fixed} basis parameter. For an
estimator that selects $\rho$ from the data, first-order asymptotics survive only under strict
conditions. These are a fixed or controlled refining grid of candidates, a unique minimizer of the
population criterion separated from the second-best candidate by a positive gap, convergence of
the estimates uniformly in $\rho$, a probability of correct selection tending to one, and
asymptotic inactivity of the safeguards. If the minimizer is not unique or the selection gap is
comparable to the sampling noise, the limit law becomes a post-selection mixture, and the fixed
sandwich cannot be used as the covariance of the data-driven estimator. The results of
Section~\ref{sec:numeric} do not test this, because the basis parameter is fixed there; the
deviation described in Section~\ref{sec:s3} has another source.

\section{Numerical study}\label{sec:numeric}

\subsection{Setup}\label{sec:setup}

The simulations answer two questions. Is the increment concentrated in the direction
$\Qm^{-1}\dbar$, and does it have the predicted size? Does formula~\eqref{eq:sigmahat} give
nominal coverage of confidence regions? A separate check locates where first-order agreement
ends. Table~\ref{tab:protocol} lists the checks, their design and the condition under which the
answer counts as positive.

\begin{table}[t]
\caption{Checks of Section~\ref{sec:numeric}: question, design and the condition under which the
answer counts as positive}
\label{tab:protocol}
\centering
\begin{tabular}{p{0.26\textwidth}p{0.32\textwidth}p{0.32\textwidth}}
\toprule
Question & Design & Positive answer if \tabularnewline
\midrule
\raggedright Is the increment concentrated along $\bm u$ (Sects.~\ref{sec:direction},~\ref{sec:rho-panel})? &
\raggedright $S=2$, scenarios A2--A4, $n=5000$, $M=1000$; optimal $\rho$ and a grid of fixed $\rho$ &
\raggedright along $\bm u$ the variance ratio exceeds one and agrees with the prediction; on the
complement it is close to one, and $\lambda_2/\lambda_1$ is close to zero \tabularnewline\addlinespace
\raggedright Does the structure of the moment conditions set the rank of the excess (Sect.~\ref{sec:rank-contrast})? &
\raggedright linear model without intercept, three schemes of conditions, $n=2000$, $10\,000$ replications &
\raggedright rank one with shared instruments and a sign-indefinite excess with a separate one \tabularnewline\addlinespace
\raggedright Does the trace metric agree with~\eqref{eq:gfe} (Sect.~\ref{sec:trace})? &
\raggedright nine scenarios, $S=2$, $n=5000$ &
\raggedright the prediction lies in the $95\%$ interval \tabularnewline\addlinespace
\raggedright Does~\eqref{eq:sigmahat} give nominal coverage (Sect.~\ref{sec:coverage})? &
\raggedright $S=2$ and $S=3$, $n$ from $200$ to $5000$, three covariance estimates &
\raggedright the closed form covers close to $0.95$, and the naive estimate undercovers where the
increment is nonzero \tabularnewline\addlinespace
\raggedright Does $\RD$ depend on the noise distribution (Sect.~\ref{sec:rd-num})? &
\raggedright $7$ designs, $3$ sample sizes, $15$ distributions, deterministic computation &
\raggedright identity~\eqref{eq:gfe} holds to machine precision \tabularnewline\addlinespace
\raggedright Where does first-order agreement end (Sect.~\ref{sec:s3})? &
\raggedright $S=3$, six symmetric scenarios, $\rho\in\{0.3,\,1\}$ &
\raggedright diagnostic, no criterion \tabularnewline
\bottomrule
\end{tabular}
\end{table}

The test problem is the harmonic signal~\eqref{eq:harmonic} with an envelope on a window holding
an incomplete period, in noise whose asymmetry and tails are set separately. The window is
$t\in[0.01,\,2]$, and the envelope amplitude is $0.3$. The noise follows the asymmetric
exponential power distribution~\citep{zhu2009properties}, standardized to zero mean and unit
variance, in nine scenarios (A0--A4; the suffixes L and R give the direction of asymmetry). For
$S=3$ we use the symmetric scenarios S075--S500 (the number is the shape exponent times~$100$),
because the factor $g_3$ was derived for symmetric errors. The first stage is constrained
NLS~\citep{handel2000properties}. The basis is fixed, and $\rho$ equals the optimum of the
population criterion. The sample sizes are $n\in\{200,500,1000,5000\}$ with $M=1000$
replications; confidence intervals come from a paired bootstrap with $1000$ resamples.

\subsection{Direction of the increment}\label{sec:direction}

The increment is concentrated along $\bm u=\Qm^{-1}\dbar/\lVert\cdot\rVert$, has the predicted
size and is absent on the orthogonal complement (Table~\ref{tab:dir}), as
Corollary~\ref{cor:direction} requires. To check this, we projected the empirical covariance of
the feasible estimator on $\bm u$ and on an orthonormal basis of the complement and compared it
with the covariance of the oracle estimator at the same $\rho$.

\begin{table}[t]
\caption{Along $\bm u$ the variance ratio exceeds one and agrees with the prediction; on the
orthogonal complement it is close to one. Ratio of the empirical variance of the feasible estimator
to that of the oracle estimator, $n=5000$, $M=1000$; ``predicted'' is the value from
Theorem~\ref{thm:main}. The last column is the ratio of the second eigenvalue of the excess matrix
to the first, which is zero for exact rank one}
\label{tab:dir}
\centering
\begin{tabular}{lccccc}
\toprule
Scenario & $\gor$ & \multicolumn{2}{c}{along $\bm u$} & on complement & $\lambda_2/\lambda_1$ \\
\cmidrule(lr){3-4}\cmidrule(lr){5-5}
 & & predicted & empirical & empirical & \\
\midrule
A2L & $0.9025$ & $1.0370$ & $1.0300$ & $1.0012$ & $0.011$ \\
A2R & $0.9025$ & $1.0370$ & $1.0301$ & $1.0014$ & $0.010$ \\
A4L & $0.7843$ & $1.0943$ & $1.0744$ & $1.0016$ & $0.006$ \\
A4R & $0.7843$ & $1.0943$ & $1.0756$ & $1.0021$ & $0.008$ \\
\midrule
A3L & $0.9943$ & $1.0020$ & $0.9964$ & $0.9991$ & --- \\
A3R & $0.9943$ & $1.0020$ & $0.9960$ & $0.9991$ & --- \\
\bottomrule
\end{tabular}
\end{table}

In scenarios A2 and A4 the ratio along $\bm u$ ranges from $1.030$ to $1.076$, and on the
complement it does not exceed $1.002$. The second eigenvalue of the excess is at most $1.1\%$ of
the first, which agrees with rank one. In scenarios A3 the predicted increment is only $0.2\%$.
Both ratios there are indistinguishable from one, and the eigenvalues of the excess have no
meaningful interpretation, so we do not report them.

\subsection{Direction of the increment at fixed $\rho$}\label{sec:rho-panel}

The direction and size of the increment are not tied to the choice of $\rho$. At every fixed
value of the basis parameter the predicted ratio along $\bm u$ lies in the $95\%$ interval,
whereas the oracle value one lies outside it (Fig.~\ref{fig:rho-panel}). This check was needed because
Theorem~\ref{thm:main} is proved for \emph{any} fixed $\rho$ outside the coalescence band,
whereas Table~\ref{tab:dir} covers only the optimal one. On the grid
$\rho\in\{0,\,0.1,\,\dots,\,0.4,\,0.6,\,\dots,\,1\}$ the same replications serve all values of
$\rho$: one NLS estimate per replication, and the feasible and oracle steps on the same data,
with $n=5000$ and $M=1000$. The interval is a paired bootstrap over replications.

\begin{figure}[t]
\centering
\includegraphics[width=\textwidth]{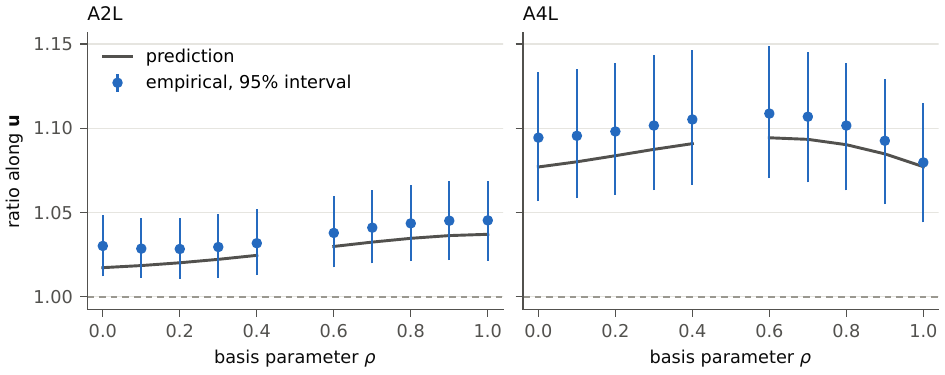}
\caption{The predicted ratio lies in the $95\%$ interval at every fixed $\rho$, whereas the
oracle value one lies outside it. Ratio of the variance of the feasible estimator to that of the oracle estimator along
$\bm u$, $S=2$, $n=5000$, $M=1000$. The line is the value from Theorem~\ref{thm:main}, the points
are empirical values with $95\%$ paired-bootstrap intervals, and the dashed line marks one. The
point $\rho=1/2$, where the basis degenerates, is excluded. The mirror scenarios A2R and A4R,
simulated with the same random numbers, differ from those shown by at most $0.0012$}
\label{fig:rho-panel}
\end{figure}

The condition holds in all $40$ cells (four scenarios times ten values of $\rho$), and the lower
bound of the interval never falls below $1.0107$. On the orthogonal complement the ratio lies
within $0.9885$--$1.0071$, and $\lambda_2/\lambda_1$ does not exceed $0.102$. The largest value
occurs for A2 at small $\rho$, where the increment itself is smallest. The empirical value exceeds
the predicted one in all cells. All cells, however, were computed on the same replications, so
this is one realization rather than forty independent ones; in another run (Table~\ref{tab:dir})
the empirical values were below the predicted ones.

\subsection{Rank of the excess under different schemes of moment conditions}\label{sec:rank-contrast}

The structure of the estimating equations determines the rank of the excess. As long as all
basis components share instruments, the rank is one, and a separate instrument for one component
makes the excess sign-indefinite. Remark~\ref{rem:rank} distinguishes three ways of entering the
centering means into the estimating equations: aggregated, through the scalar
$\bh^{\!\top}\bmm$; componentwise, with shared instruments and the efficient weight; and a scheme
in which one basis component has its own instrument. The contrast does not depend on the signal
model, so we checked it on a linear model without intercept with two regressors, the basis
$\veps,\veps^2,\veps^3$ and the centered exponential law ($n=2000$, $10\,000$ replications, two
seeds). For the aggregated and componentwise variants $|\lambda_2/\lambda_1|$ does not exceed
$0.001$. For the variant with a separate instrument for $\veps^3$ the ratio ranges from $-0.060$
to $-0.057$. The second eigenvalue is negative, so in one direction the feasible estimator is
more accurate than the oracle one.

\subsection{Prediction against simulation in the trace metric}\label{sec:trace}

The trace metric is consistent with the theorem, but at $M=1000$ it barely distinguishes the
theorem from the oracle alternative (Table~\ref{tab:trace}). Prediction~\eqref{eq:gfe} lies in
the bootstrap interval in all nine scenarios. In scenarios A2 and A4, however, where the
increment is visible, the interval is $0.05$--$0.07$ wide, and in eight of the nine scenarios it
also contains the oracle value $\gor$. The main confirmation of the theorem is therefore the
direction check (Section~\ref{sec:direction}); the trace metric merely agrees with it.

\begin{table}[t]
\caption{Prediction~\eqref{eq:gfe} lies in the bootstrap interval in all nine scenarios, but so
does the oracle value, except in A4R (last column). Variance reduction factor at $S=2$,
$n=5000$, $M=1000$; the prediction uses~\eqref{eq:gfe} with $\RD=0.2191$}
\label{tab:trace}
\centering
\begin{tabular}{lccccc}
\toprule
Scenario & $\gor$ & predicted $\gfe$ & empirical & $95\%$ interval & $\gor$ in interval? \\
\midrule
A0  & $1.0000$ & $1.0000$ & $1.0043$ & $[0.9977,\,1.0103]$ & yes \\
A1L & $0.9940$ & $0.9953$ & $0.9985$ & $[0.9867,\,1.0104]$ & yes \\
A1R & $0.9940$ & $0.9953$ & $0.9982$ & $[0.9865,\,1.0108]$ & yes \\
A2L & $0.9025$ & $0.9239$ & $0.9190$ & $[0.8927,\,0.9466]$ & yes \\
A2R & $0.9025$ & $0.9239$ & $0.9208$ & $[0.8978,\,0.9472]$ & yes \\
A3L & $0.9943$ & $0.9956$ & $0.9923$ & $[0.9857,\,0.9993]$ & yes \\
A3R & $0.9943$ & $0.9956$ & $0.9929$ & $[0.9860,\,1.0000]$ & yes \\
A4L & $0.7843$ & $0.8316$ & $0.8180$ & $[0.7831,\,0.8525]$ & yes \\
A4R & $0.7843$ & $0.8316$ & $0.8221$ & $[0.7870,\,0.8594]$ & no \\
\bottomrule
\end{tabular}
\end{table}

\subsection{Coverage of confidence regions}\label{sec:coverage}

Where the increment is nonzero, the naive covariance estimate covers less often than closed
form~\eqref{eq:sigmahat}, and the shortfall grows with the increment; where there is no
increment, the two forms coincide. Noise symmetry does not change this
(Tables~\ref{tab:coverage} and~\ref{tab:coverage-s3}).

We compare three covariance estimates. The first is the closed form, the practical consequence of
the theorem: for a fixed $\rho$ and centering on the same sample,
\begin{equation}\label{eq:sigmahat}
\hat\Sfe=\hat\Qm^{-1}/\hat\kap
+\Bigl(\hat c_2-\frac1{\hat\kap}\Bigr)\hat\Qm^{-1}\hat\dbar\hat\dbar^{\!\top}\hat\Qm^{-1},
\qquad
\hat\Qm=\frac1n\sum_i\hat\bd_i\hat\bd_i^{\!\top},\quad
\hat\dbar=\frac1n\sum_i\hat\bd_i,
\end{equation}
where $\hat\bd_i$ is the signal gradient at $\bthfe$ and $\hat\kap$, $\hat c_2$ come from step~3
of Section~\ref{sec:feasible}; the covariance of the estimate is $\hat\Sfe/n$. The second is the empirical sandwich, which
does not impose the identities of the optimal score. It is an alternative whose finite-sample
coverage is assessed below:
\begin{equation}\label{eq:sigmaemp}
\hat{\bm\chi}_i=(\hat\bd_i-\hat\dbar)\hat\psi_i+\hat\dbar\,\hat\kap\,\check\veps_i,
\qquad
\hat{\bm B}=\frac1n\sum_i\hat{\bm\chi}_i\hat{\bm\chi}_i^{\!\top},
\qquad
\hat{\bm\Sigma}_{\mathrm{emp}}=\hat\Qm^{-1}\hat{\bm B}\hat\Qm^{-1}/\hat\kap^{2},
\end{equation}
where $\check\veps_i=y_i-r_i(\bthfe)$ are the residuals of the final estimate and
$\hat\psi_i=\hat\bh^{\!\top}[\bphi(\check\veps_i)-\hat\bmm]$.
The third is the naive estimate $\hat\Qm^{-1}/\hat\kap$, which omits the increment term.

\begin{table}[t]
\caption{With a nonzero increment (A2L, A4L) the naive estimate covers worse than the closed form
at $n\ge1000$; without an increment or with a negligible one (A0, A3L) the two coincide. Empirical coverage of the joint
$95\%$ confidence region ($\chi^2$ with three degrees of freedom), $M=1000$ replications per
cell. ``Closed'' is~\eqref{eq:sigmahat}, ``empirical'' is~\eqref{eq:sigmaemp}, and ``naive'' is
$\hat\Qm^{-1}/\hat\kap$. The Monte Carlo standard error at coverage $0.95$ is $0.007$. Bold marks
the form whose coverage is closest to $0.95$, together with every form within two standard errors
($0.014$) of it; a row in which all three forms lie within this band has no bold. In three cells
at $n=200$ (A0, A2L, A4L) one or two replications were discarded because the moment specification
was inadmissible; all other cells use all $1000$}
\label{tab:coverage}
\centering
\begin{tabular}{llccc}
\toprule
Scenario & $n$ & closed \eqref{eq:sigmahat} & empirical \eqref{eq:sigmaemp} & naive \\
\midrule
A0 (increment $0$)       & $200$ & $\mathbf{0.898}$ & $0.879$ & $\mathbf{0.892}$ \\
                         & $500$ & $0.929$ & $0.923$ & $0.928$ \\
                         & $1000$ & $0.948$ & $0.942$ & $0.948$ \\
                         & $5000$ & $0.936$ & $0.939$ & $0.936$ \\
\midrule
A3L (increment $0.001$)  & $200$ & $\mathbf{0.924}$ & $0.906$ & $\mathbf{0.923}$ \\
                         & $500$ & $0.941$ & $0.942$ & $0.941$ \\
                         & $1000$ & $0.937$ & $0.937$ & $0.937$ \\
                         & $5000$ & $0.951$ & $0.949$ & $0.951$ \\
\midrule
A2L (increment $0.021$)  & $200$ & $\mathbf{0.925}$ & $0.904$ & $\mathbf{0.918}$ \\
                         & $500$ & $0.939$ & $0.938$ & $0.932$ \\
                         & $1000$ & $0.954$ & $0.952$ & $0.948$ \\
                         & $5000$ & $0.951$ & $0.952$ & $0.939$ \\
\midrule
A4L (increment $0.047$)  & $200$ & $0.923$ & $0.910$ & $0.909$ \\
                         & $500$ & $\mathbf{0.944}$ & $\mathbf{0.943}$ & $0.929$ \\
                         & $1000$ & $0.951$ & $0.951$ & $0.938$ \\
                         & $5000$ & $\mathbf{0.946}$ & $\mathbf{0.946}$ & $0.931$ \\
\bottomrule
\end{tabular}

\end{table}

At $S=2$, in the scenario without an increment (A0) and in the one with a negligible increment
(A3L), the closed and naive forms give the \emph{same} coverage to three decimals at $n\ge1000$.
The correction is not needed there, and it does not degrade coverage. In scenarios A2L and A4L
at $n\ge1000$ the coverage of the naive estimate is lower by $0.006$--$0.015$. Each of these
differences alone amounts to one or two Monte Carlo standard errors. What carries weight is that
the sign is the same in all four cells, whereas in the scenarios without an increment the
difference is zero. At $n=200$ all three estimates undercover ($0.879$--$0.925$), the empirical
form~\eqref{eq:sigmaemp} most of all. At this sample size the finite-sample error of the
refinement step and of the sandwich itself dominates, rather than the choice of covariance
formula.

\begin{table}[t]
\caption{At $n=5000$ the closed form holds the nominal level in all four symmetric scenarios, and
the shortfall of the naive estimate grows with the increment. Coverage of the joint $95\%$
confidence region at $S=3$ in the symmetric scenarios, $M=1000$ replications per cell; column
labels and the bold rule as in Table~\ref{tab:coverage}. ``Increment'' is $(1-\gor)\RD$ in the
trace metric. At $n=200$, $976$--$997$ replications were accepted, at $n=500$, $999$--$1000$,
and all $1000$ at larger $n$}
\label{tab:coverage-s3}
\centering
\begin{tabular}{llccc}
\toprule
Scenario & $n$ & closed \eqref{eq:sigmahat} & empirical \eqref{eq:sigmaemp} & naive \\
\midrule
S300 (increment $0.026$) & $200$ & $\mathbf{0.907}$ & $0.880$ & $\mathbf{0.899}$ \\
                         & $500$ & $0.932$ & $0.919$ & $0.921$ \\
                         & $1000$ & $0.931$ & $0.923$ & $0.922$ \\
                         & $5000$ & $0.957$ & $0.958$ & $0.954$ \\
\midrule
S500 (increment $0.086$) & $200$ & $\mathbf{0.914}$ & $0.869$ & $0.878$ \\
                         & $500$ & $\mathbf{0.933}$ & $0.910$ & $0.882$ \\
                         & $1000$ & $\mathbf{0.946}$ & $0.930$ & $0.901$ \\
                         & $5000$ & $\mathbf{0.954}$ & $\mathbf{0.951}$ & $0.908$ \\
\midrule
S100 (increment $0.107$) & $200$ & $\mathbf{0.704}$ & $\mathbf{0.693}$ & $0.588$ \\
                         & $500$ & $\mathbf{0.846}$ & $\mathbf{0.849}$ & $0.772$ \\
                         & $1000$ & $\mathbf{0.888}$ & $\mathbf{0.894}$ & $0.820$ \\
                         & $5000$ & $\mathbf{0.943}$ & $\mathbf{0.941}$ & $0.884$ \\
\midrule
S075 (increment $0.169$) & $200$ & $\mathbf{0.766}$ & $\mathbf{0.761}$ & $0.651$ \\
                         & $500$ & $\mathbf{0.885}$ & $\mathbf{0.891}$ & $0.770$ \\
                         & $1000$ & $\mathbf{0.944}$ & $\mathbf{0.945}$ & $0.835$ \\
                         & $5000$ & $\mathbf{0.956}$ & $\mathbf{0.958}$ & $0.804$ \\
\bottomrule
\end{tabular}

\end{table}

At $S=3$, where all noise distributions are symmetric, the closed form gives $0.943$--$0.957$ at
$n=5000$ in all four scenarios. The coverage of the naive estimate falls short by an amount that
grows with the predicted increment, from $0.003$ in S300 to $0.152$ in S075; across the four
scenarios the correlation between the shortfall and the increment is $0.973$, a descriptive
figure computed from four points. If noise symmetry
made the increment vanish, all four shortfalls would be zero (Remark~\ref{rem:symmetry}).
Finite-sample effects are stronger at $S=3$ than at $S=2$. At $n\le500$ all three forms
undercover in scenarios S075 and S100 ($0.588$--$0.891$), and at $n=1000$ the closed form gives
only $0.888$ in S100. Nominal coverage of the closed form at $S=3$ is thus established here only
for $n=5000$.

\subsection{The design nature of the weight $\RD$}\label{sec:rd-num}

The weight $\RD$ does not depend on the noise distribution: the identity
$\bigl(\gfe-\gor\bigr)/(1-\gor)=\RD$ holds to machine precision on a grid of $7$~designs $\times$
$3$~sample sizes $\times$ $15$~noise distributions. In $273$ of the $315$ cells the ratio is
defined, and its maximum deviation from $\RD$ is $1.9\cdot10^{-14}$. In the remaining $42$ cells
$\gor=1$. There the ratio is undefined, and the identity itself is trivial, because both sides
of~\eqref{eq:gfe} equal one. Table~\ref{tab:rd} gives $\RD$ for the individual designs.

\subsection{Deviation at $S=3$}\label{sec:s3}

At $S=3$ first-order agreement is not reached within the available sample sizes: the prediction
lies outside the bootstrap interval in all six scenarios (Table~\ref{tab:s3}). The sign of the gap
is set by the basis parameter, not by the noise distribution. In the four scenarios with
$\rho=0.3$ (S075--S200) the empirical value exceeds the prediction, and in the two with $\rho=1$
(S300, S500) it falls below it. Two different mechanisms are therefore at work, not one
phenomenon with a varying sign. The S200 row does not test the theorem. There $\gor=1$,
increment~\eqref{eq:main} is identically zero and $\gfe=1$ for any $\RD$, so the gap describes
the estimator itself.

\begin{table}[t]
\caption{The sign of the gap is set by $\rho$, not by the noise distribution. Variance reduction
factor at $S=3$, $n=5000$, $M=1000$, fixed basis parameter; ``gap'' is the empirical value minus
the prediction. In the S200 row $\gor=1$ and the prediction is trivial, so this row describes the
estimator rather than testing the theorem}
\label{tab:s3}
\centering
\begin{tabular}{lccccc}
\toprule
Scenario & $\rho$ & $\gor$ & predicted $\gfe$ & empirical & gap \\
\midrule
S075 & $0.3$ & $0.2273$ & $0.3966$ & $0.4352$ & $+0.0386$ \\
S100 & $0.3$ & $0.5108$ & $0.6180$ & $0.6879$ & $+0.0700$ \\
S150 & $0.3$ & $0.9126$ & $0.9318$ & $1.0228$ & $+0.0911$ \\
S200 & $0.3$ & $1.0000$ & $1.0000$ & $1.0886$ & $+0.0886$ \\
S300 & $1$ & $0.8833$ & $0.9089$ & $0.8823$ & $-0.0266$ \\
S500 & $1$ & $0.6085$ & $0.6943$ & $0.6498$ & $-0.0444$ \\
\bottomrule
\end{tabular}
\end{table}

The gap in S200 is not caused by the refinement step. The oracle estimator with the same step
and the same damping reproduces $g=1$ there to within $6\cdot10^{-8}$, with a mean damping
factor of $1.0$. The classical polynomial maximization method at $S=3$ (integer-power basis) with
the same plug-in mechanism gives $0.9992$. The whole gap of $0.0886$ in this row therefore comes
from estimating the weights of the fractional basis. The choice of $\rho$ plays no part in it,
because the basis parameter is fixed here.

The source of the gap is the finite-sample bias of the \emph{weight} estimates for bases with
$\rho<1/2$. At $\rho=0.3$ the vector $\bb$ requires absolute moments of negative orders
($q_2-1\approx-0.26$, $q_3-1\approx-0.43$), and $q_3$ lies below the bound $2/3$ above which the
smoothing bias is negligible (Section~\ref{sec:limits}). The estimate $\hat g$ produced by the
estimator itself therefore deviates from the population value by $0.07$--$0.12$ at $\rho=0.3$
and by less than $0.002$ at $\rho=1$.

A diagnostic computation supports this mechanism (Fig.~\ref{fig:s3-sweep}). The excess over the
prediction is large for $\rho<1/2$ and small or negative for $\rho>1/2$, where the negative
moment orders disappear. The computation used the grid $\rho\in\{0,\,0.1,\,\dots,\,1\}$ without
the degeneracy point $\rho=1/2$, in scenarios S200 (where $\gor=1$ for any $\rho$) and S100, with
$250$ replications and $n\in\{500,2000,8000\}$. The excess correlates with the absolute bias of
$\hat g$ ($0.71$--$0.97$) and hardly correlates with the logarithm of the condition number of
$\Fm$ (from $-0.20$ to $0.23$). It is thus caused by a systematic error of the weights, not by
conditioning. Growth of $n$ does not guarantee that the excess disappears: in S100 at
$\rho=0.3$ it does not decrease between $n=2000$ and $n=8000$.

\begin{figure}[t]
\centering
\includegraphics[width=\textwidth]{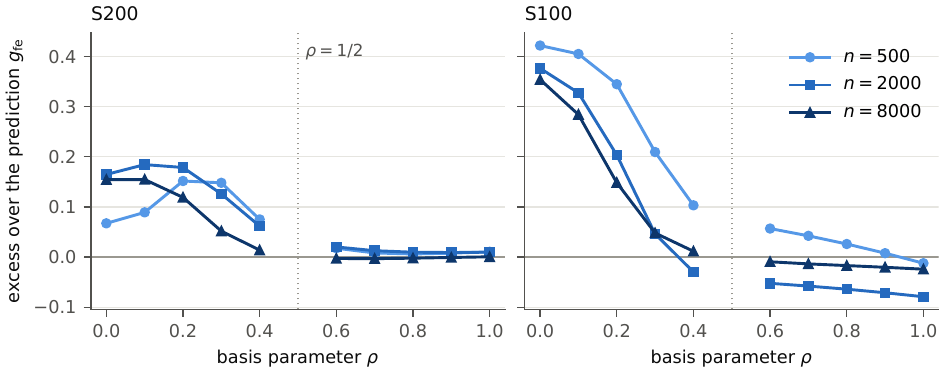}
\caption{The excess is large only for $\rho<1/2$, where the vector $\bb$ needs moments of
negative orders. Difference between the empirical variance reduction factor and the
prediction~$\gfe$ at $S=3$ for fixed $\rho$, $250$ replications (diagnostic run); the point
$\rho=1/2$, where the basis degenerates, is excluded. In S200, $\gor=1$ for any $\rho$}
\label{fig:s3-sweep}
\end{figure}

This does not contradict Theorem~\ref{thm:main}. The theorem concerns the block of centering
means, whereas the gap comes from the block of weights, which is orthogonal by~\eqref{eq:orth}.
Orthogonality is a first-order statement and does not bound the size of the remainder. At
$\rho=0.3$ and $n=5000$, near the Gaussian scenario, this remainder exceeds the first-order
increment that is the subject of the paper: in S150 the gap is $0.091$ against an increment of
$0.019$. In S075 and S100 it is between a quarter and two thirds of the increment
(Table~\ref{tab:s3}). This paper does not propose a correction that would remove the weight bias
at $\rho<1/2$. In the two scenarios with $\rho=1$ the gap is negative. Part of it comes from the
oracle estimator already, whose factor there lies below the population value by $0.020$ and
$0.027$ (Section~\ref{sec:caveat}).

The gap in the trace metric does not invalidate covariance estimate~\eqref{eq:sigmahat}. Computed
with the estimated $\hat\kap$ and $\hat c_2$, it reproduces the actual spread of the estimates at
$n=5000$ in the same symmetric scenarios. The ratio of the mean standard error to the Monte Carlo
spread lies within $0.96$--$1.06$ for all three parameters, and the coverage is nominal
(Table~\ref{tab:coverage-s3}). The gap concerns the comparison with the \emph{population}
prediction, not the internal consistency of the covariance estimate.

\subsection{A caveat on the reference level}\label{sec:caveat}

The rank-one identity concerns asymptotic covariances. Table~\ref{tab:dir} compares the feasible
estimator with the \emph{empirical} oracle estimator using the same refinement step and
population moments; at finite sample size this is a check of the asymptotic prediction, not an
exact rank-one identity for empirical covariance matrices. Relative to the closed form~\eqref{eq:oracle-influence}, both estimators, the
feasible and the oracle one, have an isotropic variance excess of about $1$--$5\%$ at $n=5000$.
It comes from the finite-sample behavior of one refinement step from the NLS estimate under
non-Gaussian laws. In the Gaussian scenario, where the optimal score is linear, the same step is
exact (Section~\ref{sec:s3}). This excess has no distinguished direction and so does not affect
the check of Corollary~\ref{cor:direction}, but it should be kept in mind when interpreting
absolute values.

\section{Discussion}\label{sec:discussion}

The distinction from the earlier PMM--SLS result is between a slope equation from which
estimated centering cancels and a full parameter vector for which it can survive
(Section~\ref{sec:reconcile}). The present contribution is the explicit correction in the
latter setting. Consistency and oracle equivalence are different assertions: when
$\gor<1$ and $\dbar\ne\bm0$, increasing the sample size reduces both covariances by the same
factor but leaves their relative difference. Oracle equivalence can follow from a zero
nuisance derivative, a vanishing nuisance influence, or externally known centering means
(Corollary~\ref{cor:equiv}). These are distinct mechanisms. Symmetry alone does not enforce
any of them for the implemented estimator (Remark~\ref{rem:symmetry}).

The rank bound follows from the scalar aggregation of the centering error, but positivity
requires more: that aggregate is uncorrelated with the oracle influence. The score identities
then give its variance $c_2(1-\gor)$. General two-step theory supplies the expansion; the
rank, sign, direction and closed coefficient together describe what that expansion means for
this estimator. The contrast with different instruments in Section~\ref{sec:rank-contrast}
shows why positivity cannot be attributed to nuisance estimation in general.

For an individual parameter, its coordinate in $\Qm^{-1}\dbar$ decides whether oracle
precision survives. With an additive intercept only its estimate bears the loss. This does
not justify adding a physically unwarranted intercept merely to claim oracle precision:
changing the model also changes its oracle benchmark and its design matrix. Without an
intercept, contrasts orthogonal to the correction direction remain unaffected, whereas other
contrasts require the additional variance. Corollary~\ref{cor:wald} makes the inference
consequence explicit: the naive Wald statistic has a stretched normal or weighted
chi-squared limit. Its undercoverage therefore need not disappear with more observations.
The corrected covariance~\eqref{eq:sigmahat} accounts for this first-order loss;
the empirical alternative~\eqref{eq:sigmaemp} has no general finite-sample superiority.

Design can reduce the centering loss, but the criterion must match the scientific target.
For a specified contrast its additional variance follows from Corollary~\ref{cor:direction}.
The trace weight $\RD$ in~\eqref{eq:gfe} combines parameter coordinates with their chosen
scales; it should not be treated as invariant to changing the units of those parameters.
The squared projection length $\pi=\dbar^{\!\top}\Qm^{-1}\dbar$ is invariant under nonsingular
linear reparameterization and controls the joint Wald distortion.

The results have several limitations. The theorem is proved for a general smooth signal model,
but simulations confirm it only for the harmonic signal with an envelope; for the other models of
Table~\ref{tab:classes} we checked algebraic identities, not the finite-sample behavior of the
estimator. The basis shape was taken as given a priori, and selecting it from the data requires
post-selection theory or a bootstrap of the whole estimation procedure
(Section~\ref{sec:limits}). Only independent identically distributed noise was considered
(item~3 of Assumption~\ref{ass:reg}). The description of correlated
noise~\citep{vokorokos2017parameters} involves two-time cumulant functions, so the influence
expansion given here does not extend to that case. No real measurements are used: the subject of
the paper is an asymptotic statement, which we tested in a controlled computational experiment.
Finally, at $S=3$ and $\rho<1/2$, the loss from estimating the weights, which lies outside
first-order asymptotics, is comparable at attainable $n$ to the loss from centering and exceeds
it near the Gaussian scenario. Section~\ref{sec:s3} describes the mechanism of this effect;
constructing a correction for it remains an open problem.

\section{Conclusions}\label{sec:conclusions}

When polynomial estimation of signal parameters passes from the oracle to the feasible scheme,
the asymptotic covariance grows by exactly one term. For a signal given by a smooth model in a
fixed design, this term is the positive semidefinite matrix
$c_2(1-\gor)\Qm^{-1}\dbar\dbar^{\!\top}\Qm^{-1}$ of rank at most one. Estimating the fractional
moments that form the weights does not affect first-order asymptotics; the increment comes only
from estimating the sign-preserving centering moments. The increment is concentrated in the
single direction $\Qm^{-1}\dbar$, and in the trace metric its weight $\RD$ is set by the
experimental design and does not depend on the noise distribution. The vector $\Qm^{-1}\dbar$
consists of the coefficients of the projection of the constant onto the components of the signal
gradient. If the model contains a constant term, the increment falls on its estimate alone.
Simulations confirm the result for the harmonic signal with an envelope.


\section*{Statements and Declarations}

\noindent\textbf{Funding.} No funding was received for this work.\\
\textbf{Competing interests.} The author declares no competing interests.\\
\textbf{Ethics approval and consent to participate.} Not applicable: the study involves no human
or animal subjects.\\
\textbf{Data availability.} All results were obtained from synthetic Monte Carlo data
(Section~\ref{sec:numeric}); no external data were used.\\
\textbf{Code availability.} The code that reproduces every number, table and figure of the paper,
together with the stored results, is openly available at
\url{https://github.com/SZabolotnii/Ku-PMM-Rank-One-code-supplement} (MIT licence).\\
\textbf{Use of AI-based tools.} AI-based tools (Claude Code, Codex) assisted in implementing and checking
the code, in reconciling the numbers of the manuscript with the stored results, in drafting the
text and translating it from Ukrainian, and in language editing and \LaTeX{} formatting. The
author reviewed and edited the whole text and takes full responsibility for the content of the
article.

\appendix

\section{Proofs}\label{app:proofs}

\subsection{Influence of the estimated centering moments}\label{app:lem-mm}

\begin{proof}[Proof of Lemma~\ref{lem:mm}]
Let $\bm{D}_i=\bd_i-\dbar_n$, so that $\sum_i\bm{D}_i=\bm 0$ by construction, and
$\bar\veps=n^{-1}\sum_i\veps_i$. By item~7 of Assumption~\ref{ass:reg},
$\sqrt n(\bthtil-\bthz)=O_p(1)$, and by item~1 the second design derivatives are uniformly
bounded, so that, uniformly in~$i$,
\[
\hat\veps_i=\veps_i-\bd_i^{\!\top}(\bthtil-\bthz)+O_p(n^{-1}),
\qquad
\hat\veps_i^{c}=(\veps_i-\bar\veps)-\bm{D}_i^{\!\top}(\bthtil-\bthz)+O_p(n^{-1}).
\]
Expanding $\bphi(\hat\veps_i^{c})$ about $\bphi(\veps_i)$ and averaging, we obtain
\begin{equation}\label{eq:mm-expand}
\begin{aligned}
\sqrt n(\hat\bmm-\bmm)&=\frac1{\sqrt n}\sum_i\bz(\veps_i)
-\Bigl[\frac1n\sum_i\bphi'(\veps_i)\Bigr]\sqrt n\,\bar\veps
-\bm{A}_n\sqrt n(\bthtil-\bthz)+o_p(1),\\
\bm{A}_n&=\frac1n\sum_i\bphi'(\veps_i)\bm{D}_i^{\!\top},
\end{aligned}
\end{equation}
where item~8 makes the remainder uniform. The coefficient of $(\bthtil-\bthz)$ is the
\emph{empirical} $\bm{A}_n$, not $\bb\,n^{-1}\sum_i\bm{D}_i^{\!\top}$: the identity
$\sum_i\bm{D}_i=\bm 0$ annihilates only the population part. We split
\[
\bm{A}_n=\bb\,\frac1n\sum_i\bm{D}_i^{\!\top}
+\underbrace{\frac1n\sum_i\bigl[\bphi'(\veps_i)-\bb\bigr]\bm{D}_i^{\!\top}}_{=\;\tilde{\bm{A}}_n}
=\tilde{\bm{A}}_n .
\]
Since $\varphi_j'(\veps)=q_j|\veps|^{q_j-1}$ for $\veps\ne0$, there is no atom at zero and
$\E|\veps|^{q_j-1}<\infty$ (item~6), we have $\E[\varphi_j'(\veps)]=q_j\E|\veps|^{q_j-1}=b_j$.
This coincides with the definition of $\bb$ in~\eqref{eq:score}, so no differentiation under the
expectation is needed. The terms $\bm{w}_i=\bphi'(\veps_i)-\bb$ are therefore independent and
identically distributed, with zero mean and finite first absolute moment, and the weights
$\bm{D}_i$ are uniformly bounded by item~1. For such terms $\tilde{\bm{A}}_n\pto\bm 0$. Under the
truncation
$\bm{w}_i^{M}=\bm{w}_i\mathbf 1\{\lVert\bm{w}_i\rVert\le M\}-\E\bigl[\bm{w}\mathbf 1
\{\lVert\bm{w}\rVert\le M\}\bigr]$ the bounded part has variance of order $O(n^{-1})$
(Chebyshev's inequality), and the tail is bounded in $L_1$ by
$2\sup_i\lVert\bm{D}_i\rVert\,\E\bigl[\lVert\bm{w}\rVert\mathbf 1\{\lVert\bm{w}\rVert>M\}\bigr]$,
which tends to zero as $M\to\infty$ by dominated convergence. Together with
$\sqrt n(\bthtil-\bthz)=O_p(1)$ this gives $\bm{A}_n\sqrt n(\bthtil-\bthz)=o_p(1)$. The
contribution of the preliminary NLS estimate vanishes at this step, not through an algebraic
identity.

The second term in~\eqref{eq:mm-expand} is handled by the law of large numbers:
$n^{-1}\sum_i\bphi'(\veps_i)\pto\bb$ and $\sqrt n\,\bar\veps=n^{-1/2}\sum_i\veps_i=O_p(1)$, so
this term equals $\bb\,n^{-1/2}\sum_i\veps_i+o_p(1)$. Substituting both results
into~\eqref{eq:mm-expand} gives~\eqref{eq:mm-influence}. Centering the residuals on the same
sample thus asymptotically orthogonalizes the signed-moment estimator with respect to the
unknown location. For $q_1=1$ we have $\varphi_1(\veps_i)=\veps_i$, $m_1=\E\veps=0$ and $b_1=1$,
so $z_1(\veps_i)-b_1\veps_i=0$ exactly.
\end{proof}

\subsection{Difference between the feasible and oracle estimators}\label{app:diff}

Representation~\eqref{eq:diff-influence}, on which Remark~\ref{rem:dependence} and the proofs of
both corollaries rely, is derived here separately.

\begin{lemma}\label{lem:diff}
Under the conditions of Theorem~\ref{thm:main}, the difference between the feasible and oracle
estimators has influence representation~\eqref{eq:diff-influence}. The variable
$W=\veps-\psi(\veps)/\kap$ has $\E W=0$ and $\Var W=c_2-1/\kap$ and is uncorrelated with
$\psi(\veps)$. Hence influence~\eqref{eq:diff-influence} is asymptotically uncorrelated with the
oracle influence~\eqref{eq:oracle-influence}, and its asymptotic covariance is
$(c_2-1/\kap)\,\Qm^{-1}\dbar\dbar^{\!\top}\Qm^{-1}$.
\end{lemma}

\begin{proof}
Subtract~\eqref{eq:oracle-influence} from~\eqref{eq:feas-influence}. For each $i$,
\[
\bm\chi_i-\bd_i\,\psi(\veps_i)=\dbar\,\bigl[\kap\,\veps_i-\psi(\veps_i)\bigr]=\kap\,\dbar\,W_i,
\qquad
W_i=\veps_i-\frac{\psi(\veps_i)}{\kap},
\]
and the factor $\Qm^{-1}\kap^{-1}$ gives~\eqref{eq:diff-influence}. Next, $\E W=0$, because
$\E\veps=0$ and $\E\psi(\veps)=\bh^{\!\top}\E[\bz(\veps)]=0$. By
identities~\eqref{eq:identities},
\begin{align*}
\Cov[\psi(\veps),W]&=\Cov[\psi(\veps),\veps]-\frac{\Var\psi(\veps)}{\kap}=1-1=0,\\
\Var W&=c_2-\frac{2}{\kap}\Cov[\psi(\veps),\veps]+\frac{\Var\psi(\veps)}{\kap^{2}}
=c_2-\frac{1}{\kap}.
\end{align*}
The errors are independent (item~3 of Assumption~\ref{ass:reg}), so the cross-covariance of the
sums $n^{-1/2}\sum_i\bd_i\psi(\veps_i)$ and $n^{-1/2}\sum_i W_i$ equals
$n^{-1}\sum_i\bd_i\Cov[\psi(\veps),W]=\bm 0$, and the covariance of the right-hand side
of~\eqref{eq:diff-influence} is $\Var W\cdot\Qm^{-1}\dbar\dbar^{\!\top}\Qm^{-1}$. This is the
term in~\eqref{eq:main}, in agreement with Theorem~\ref{thm:main}: the difference is
uncorrelated with the oracle influence, so $\Sfe$ is the sum of $\Sor$ and its covariance.
\end{proof}

\subsection{Direction of the increment}\label{app:direction}

\begin{proof}[Proof of Corollary~\ref{cor:direction}]
By item~2 of Assumption~\ref{ass:reg} the matrix $\Qm$ is symmetric and positive definite, so
$\Qm^{-1}$ is symmetric and $\Qm^{-1}\dbar\dbar^{\!\top}\Qm^{-1}=(\Qm^{-1}\dbar)(\Qm^{-1}\dbar)^{\!\top}$.
Substituting this into~\eqref{eq:main}, for any $\bm{a}\in\R^p$ we obtain
\[
\bm{a}^{\!\top}(\Sfe-\Sor)\bm{a}
=c_2(1-\gor)\,\bm{a}^{\!\top}(\Qm^{-1}\dbar)(\Qm^{-1}\dbar)^{\!\top}\bm{a}
=c_2(1-\gor)\,\bigl(\bm{a}^{\!\top}\Qm^{-1}\dbar\bigr)^{2}.
\]
The factor $c_2(1-\gor)$ is nonnegative by Theorem~\ref{thm:main}, so the form is nonnegative;
it vanishes if and only if $\gor=1$ \emph{or} $\bm{a}^{\!\top}\Qm^{-1}\dbar=0$. In particular,
for $\gor<1$ the factor is strictly positive, and the zero is attained exactly on the orthogonal
complement of $\Qm^{-1}\dbar$; for $\gor=1$ the form vanishes for all $\bm{a}$, because the whole
increment matrix is zero.

The second statement is stronger than equality of asymptotic variances and follows directly
from~\eqref{eq:diff-influence}: for any $\bm{a}$ with $\bm{a}^{\!\top}\Qm^{-1}\dbar=0$,
\[
\sqrt n\,\bm{a}^{\!\top}(\bthfe-\bthor)
=\bigl(\bm{a}^{\!\top}\Qm^{-1}\dbar\bigr)\frac1{\sqrt n}\sum_{i=1}^n
\Bigl[\veps_i-\frac{\psi(\veps_i)}{\kap}\Bigr]+o_p(1)=o_p(1).
\]
Such a linear combination of the feasible and oracle estimators thus coincides to first order,
rather than merely having the same asymptotic variance.

Finally, for $\dbar=\bm 0$ or $\gor=1$ the increment matrix in~\eqref{eq:main} is zero, and the
vector $\bm u$ is undefined when $\dbar=\bm 0$. The statement about a \emph{single} direction
therefore requires $\dbar\ne\bm 0$ and $\gor<1$.
\end{proof}

\subsection{Sufficient conditions for oracle equivalence}\label{app:equiv}

\begin{proof}[Proof of Corollary~\ref{cor:equiv}]
In cases~2--4 it suffices to show that the leading term in~\eqref{eq:diff-influence} vanishes;
case~1 changes the procedure itself and is treated separately.

\emph{Case~2.} For $\dbar=\bm 0$ the factor $\Qm^{-1}\dbar$ in~\eqref{eq:diff-influence} is zero,
so the whole right-hand side is $o_p(1)$.

\emph{Case~3.} By Lemma~\ref{lem:diff}, $W=\veps-\psi(\veps)/\kap$ has $\E W=0$ and
$\Var W=c_2-1/\kap$. The condition $c_2=1/\kap$, that is, $\gor=1$ by~\eqref{eq:gor}, gives
$\Var W=0$. Hence $W=0$ almost surely and $\veps=\psi(\veps)/\kap$ in $L_2$, and the sum
in~\eqref{eq:diff-influence} is identically zero. Incidentally, $\Var W=c_2-1/\kap\ge0$ gives
another proof of the inequality $\gor\le1$, which Theorem~\ref{thm:main} derives from the
Cauchy--Schwarz inequality.

\emph{Case~4.} If $\bm{C}=\bm 0$, or if the exact cancellation
$\bm{C}\bm{H}^{-1}(\bm{v}_i+\bm{D}\bm{\ell}_i)=\bm 0$ holds, the efficient score
in~\eqref{eq:twostep} reduces to $\bm{s}_i^{\mathrm{eff}}=\bm{u}_i$, the oracle score, and the two
influence representations coincide.

\emph{Case~1.} If $\bmm$ is supplied externally, Lemma~\ref{lem:mm} is not needed: the centered
basis $\bz$ is known exactly, and the difference of the scores is
$\hat\psi(\veps)-\psi(\veps)=(\hat\bh-\bh)^{\!\top}\bz(\veps)$. The perturbation of the
estimating equation is
\[
\frac1{\sqrt n}\sum_{i=1}^n\bd_i\bigl[\hat\psi(\veps_i)-\psi(\veps_i)\bigr]
=\Bigl[\frac1{\sqrt n}\sum_{i=1}^n\bd_i\,\bz(\veps_i)^{\!\top}\Bigr](\hat\bh-\bh).
\]
The matrix in brackets has zero mean, because the design is fixed and $\E[\bz(\veps)]=\bm 0$
by~\eqref{eq:orth}. Its second-moment norm is bounded by
$\bigl(n^{-1}\sum_i\lVert\bd_i\rVert^{2}\bigr)\lVert\Fm\rVert=O(1)$ by items~1--2 and~5 of
Assumption~\ref{ass:reg}, so the matrix is $O_p(1)$. The factor $\hat\bh-\bh$ is $o_p(1)$ by
item~8, so the product is $o_p(1)$; likewise, $\hat\kap^{-1}-\kap^{-1}=o_p(1)$ multiplies a
normalized score of order $O_p(1)$. The feasible and oracle refinement steps therefore coincide up
to $o_p(1)$. This is the content of~\eqref{eq:orth}: the block $(\Fm,\bb,\bh,\kap)$ is
orthogonal, and consistency of its estimates suffices.

The last statement of the corollary needs no separate proof. The factor $c_2(1-\gor)$
in~\eqref{eq:main} does not depend on $n$, so the ratio
$\bm{a}^{\!\top}\Sfe\bm{a}\big/\bm{a}^{\!\top}\Sor\bm{a}$ does not change as the sample grows.
\end{proof}

\bibliographystyle{plainnat}
\bibliography{references}

\end{document}